\documentclass[12pt,letterpaper]{article}
\usepackage{a4, graphicx,amsmath,rotating,amssymb,epsfig,epstopdf,amsthm,etoolbox}
\usepackage[numbers]{natbib}
\usepackage{pslatex} 
\usepackage{graphicx}
\usepackage{amssymb}
\usepackage{epstopdf}
\usepackage{titlesec}
\usepackage{tikz,subfig}
\usetikzlibrary{shapes,positioning,intersections,quotes,svg.path,shapes.geometric}
\usepackage{algpseudocode}
\usepackage{algorithm}
\usepackage[mathscr]{eucal}
\usepackage[margin=1in]{geometry}
\usepackage{mathtools}
\usepackage{hyperref}
\hypersetup{
    plainpages=false,       
    unicode=false,          
    pdftoolbar=true,        
    pdfmenubar=true,        
    pdffitwindow=false,     
    pdfstartview={FitH},    
    pdftitle={uWaterloo\ LaTeX\ Thesis\ Template},    
    pdfnewwindow=true,      
    colorlinks=true,        
    linkcolor=blue,         
    citecolor=blue,        
    filecolor=magenta,      
    urlcolor=cyan           
}
\makeatletter
\patchcmd{\NAT@test}{\else \NAT@nm}{\else \NAT@hyper@{\NAT@nm}}{}{}
\makeatother

\defcitealias{tutte72}{Tutte (1972) [cf. Bondy and Murty (1976)]}
\defcitealias{tutte54}{Tutte 1954}
\defcitealias{tutte66}{1966}

\theoremstyle{plain}
\newtheorem{theorem}{Theorem}[section]
\newtheorem{lemma}[theorem]{Lemma}
\newtheorem{corollary}[theorem]{Corollary}
\newtheorem{conjecture}{Conjecture}
\newtheorem{definition}[theorem]{Definition}

\newtheorem{claim}[theorem]{Claim}
\newtheorem{prop}{}
\newenvironment{Prop}[2][]
{\renewcommand\theprop{#2}\begin{prop}[#1]}
{\end{prop}}

\newtheorem{inclaim}{Claim}
\newenvironment{Inclaim}[2][]
{\renewcommand\theinclaim{#2}\begin{inclaim}}
{\end{inclaim}}

\begin{document}

\title{Two Strong 3-Flow Theorems for Planar Graphs}
\author{J. V. de Jong\footnote{Department of Combinatorics and Optimization, University of Waterloo, Canada (jamiev.dejong@gmail.com)}}

\maketitle
\begin{abstract}
In 1972, Tutte posed the $3$-Flow Conjecture: that all $4$-edge-connected graphs have a nowhere zero $3$-flow. This was extended by Jaeger et al.~(1992) to allow vertices to have a prescribed, possibly non-zero difference (modulo $3$) between the inflow and outflow. They conjectured that all $5$-edge-connected graphs with a valid prescription function have a nowhere zero $3$-flow meeting that prescription (we call this the Strong $3$-Flow Conjecture). Kochol (2001) showed that replacing $4$-edge-connected with $5$-edge-connected would suffice to prove the $3$-Flow Conjecture and Lov\'asz et al.~(2013) showed that the $3$-Flow and Strong $3$-Flow Conjectures hold if the edge connectivity condition is relaxed to $6$-edge-connected. Both problems are still open for $5$-edge-connected graphs. \\

The $3$-Flow Conjecture was known to hold for planar graphs, as it is the dual of Gr\"otzsch's Colouring Theorem. Steinberg and Younger (1989) provided the first direct proof using flows for planar graphs, as well as a proof for projective planar graphs. Richter et al.~(2016) provided the first direct proof using flows of the Strong $3$-Flow Conjecture for planar graphs. We provide two extensions to their result, that we developed in order to prove the Strong $3$-Flow Conjecture for projective planar graphs. 
\end{abstract}

\section{Introduction}

A \emph{$\mathbb{Z}_k$-flow} on a graph $G$ is a function that assigns to each edge $e\in E(G)$ an ordered pair consisting of a direction, and a value $f(e)\in\{0,...,k-1\}$, such that if $D$ is the resulting directed graph, then, for each vertex $v\in V(G)$,
$$\sum_{e=(u,v)\in E(D)}f(e)-\sum_{e=(v,w)\in E(D)}f(e)\equiv 0\pmod{k}.$$ 
It is easy to see that every graph $G$ has a $\mathbb{Z}_k$-flow for every value of $k$: set $f(e)=0$ for all $e\in E(G)$. Therefore it is typical to use the following more restrictive concept. A \emph{nowhere zero $\mathbb{Z}_k$-flow} on $G$ is a $\mathbb{Z}_k$-flow on $G$ such that no edge is assigned the value zero. In 1950, Tutte proved that a graph has a nowhere-zero $\mathbb{Z}_k$-flow if and only if it has a nowhere-zero $k$-flow, which requires the net flow through each vertex to be exactly zero; see \citet{diestal05} for a proof of this equivalence and further background on flows.\\

Tutte (cf.~\citet{tutte72}) conjectured that every $4$-edge-connected graph has a nowhere zero $3$-flow. This is known as the $3$-Flow Conjecture, and while progress has been made for many classes of graphs, it is still an open problem. For planar graphs the $3$-Flow Conjecture is equivalent to Gr\"otzsch's Theorem, and \citet{steinbergyounger89} provided a direct proof using flows. \citet{steinbergyounger89} also proved that the $3$-Flow Conjecture holds for graphs embedded in the projective plane. \\

As an extension of $\mathbb{Z}_3$-flows, we add a prescription function, where each vertex in the graph is assigned a value in $\mathbb{Z}_3$ that defines the net flow through the vertex. The prescriptions of the vertices in the graph must sum to zero in $\mathbb{Z}_3$. A graph $G$ is $\mathbb{Z}_3$-connected if for each valid prescription function $p$, $G$ has a nowhere-zero $\mathbb{Z}_3$-flow achieving $p$. This led \citet{jaegar92} to pose the following conjecture.

\begingroup
\def\thetheorem{\ref{jaeger}}
\begin{conjecture}[Strong $3$-Flow Conjecture]
Every $5$-edge-connected graph is $\mathbb{Z}_3$-connected.
\end{conjecture}
\endgroup

\citet{jaegar88} had earlier posed the following weaker conjecture.

\begin{conjecture}[Weak 3-Flow Conjecture]
There is a natural number $h$ such that every $h$-edge-connected graph has a nowhere zero $\mathbb{Z}_3$-flow. 
\end{conjecture}

\citet{jaegar88} also showed that this conjecture is equivalent to the same statement regarding $\mathbb{Z}_3$-connectivity. The Weak 3-Flow Conjecture remained open until \citet{thomassen12} proved that $h=8$ sufficed. 

\begin{theorem}
Every $8$-edge-connected graph is $\mathbb{Z}_3$-connected. 
\end{theorem}

\citet{lovasz13} extended this to the following result. 

\begin{theorem}
If $G$ is a $6$-edge-connected graph, then $G$ is $\mathbb{Z}_3$-connected.
\end{theorem}

\citet{kochol01} showed that the $3$-Flow Conjecture is equivalent to the statement that every $5$-edge-connected graph has a nowhere zero $\mathbb{Z}_3$-flow, so the result of \citet{lovasz13} is one step away from the $3$-Flow Conjecture. As stated, both of these results also make significant steps toward the Strong $3$-Flow Conjecture, as they considered $\mathbb{Z}_3$-connectivity. \citet{laili06} proved the Strong $3$-Flow Conjecture holds for planar graphs using the duality with graph colouring. \citet{richteretal16} provided the first direct proof of this result using flows. Their result is Theorem \ref{mainplanar}. 

\begin{theorem}
\label{mainplanar}
Let $G$ be a $3$-edge-connected graph embedded in the plane with at most two specified vertices $d$ and $t$ such that
\begin{itemize}
\item if $d$ exists, then it has degree $3$, $4$, or $5$, has its incident edges oriented and labelled with elements in $\mathbb{Z}_3\setminus\{0\}$, and is in the boundary of the unbounded face,
\item if $t$ exists, then it has degree $3$ and is in the boundary of the unbounded face,
\item there are at most two $3$-cuts, which can only be $\delta(\{d\})$ and $\delta(\{t\})$,
\item if $d$ has degree $5$, then $t$ does not exist, and
\item every vertex not in the boundary of the unbounded face has five edge-disjoint paths to the boundary of the unbounded face. 
\end{itemize}
If $G$ has a valid prescription function, then $G$ has a valid $\mathbb{Z}_3$-flow. 
\end{theorem}

The $3$-Flow Conjecture and the Strong $3$-Flow Conjecture for planar graphs are corollaries of Theorem~\ref{mainplanar}.\\

Our aim was to extend the theorem of \citet{richteretal16} to projective planar graphs. In investigating this problem we required new results extending our knowledge of $\mathbb{Z}_3$-connectivity for planar graphs. These results form the basis of this paper. Our proof for the projective planar case applies these results and appears in \citet{paper2}. All of these results also appear in the author's Ph.D. thesis \citep{thesis}. The techniques used for the proofs in this paper build off those of \citet{thomassen12}, \citet{lovasz13}, and \citet{richteretal16}.  \\

We provide two extensions of Theorem \ref{mainplanar}. The first allows two unoriented degree $3$ vertices, instead of only one. This is analogous to a result of \citet{steinbergyounger89} showing that the $3$-Flow Conjecture holds for planar $3$-edge-connected graphs with at most three $3$-edge-cuts. \citet{steinbergyounger89} required this result to prove the $3$-Flow Conjecture for projective planar graphs, just as we require this extension to prove that the Strong $3$-Flow Conjecture holds for projective planar graphs.\\

The second extension allows vertices of low degree on two specified faces of the embedding that have a common vertex, provided that specified vertices $d$ and $t$ do not both exist. This will allow us to work with certain edge-cuts that arise in projective planar graphs. Optimally, such a result would be possible without these restrictions, as this would facilitate solutions to certain configurations in toroidal graphs, rather than just projective planar graphs. However, an analogous result with two specified faces that do not need to have a common vertex, as well as specified vertices $d$ and $t$, is not true, as we will see in Section \ref{projdis}.  \\

In Section \ref{projprelim} we first discuss some of the ideas that will be used throughout the proofs in this paper. 

\section{Preliminaries}
\label{projprelim}

\subsection*{Basic Definitions}
We define an \emph{orientation} of a graph $G$ to be the directed graph $D$ obtained by adding a direction to each edge. With reference to finding a $\mathbb{Z}_3$-flow on $G$ with prescription function $p$, an orientation is \emph{valid} if for each vertex $v\in V(G)$,
$$\sum_{e=(u,v)\in E(D)}1-\sum_{e=(v,w)\in E(D)}1\equiv p(v)\pmod{3}.$$ 
It can be easily seen that this is equivalent to a $\mathbb{Z}_3$-flow.\\

A vertex $d$ in a graph $G$ is a \emph{directed vertex} if all its incident edges are directed. We call this an \emph{orientation} of $d$. We say that an orientation of $G$ \emph{extends} the orientation of $d$ if the direction of the edges at $d$ is maintained. In cases involving a directed vertex we take the term \emph{valid orientation} to include that the orientation extends that of~$d$. \\

We \emph{lift} a pair of edges $uv$ and $vw$ in a graph $G$ by deleting $uv$ and $vw$, and adding an edge $vw$. We define an edge-cut $\delta(A)$ in $G$ to be \emph{internal}\index{internal edge-cut} if either $A$ or $G-A$ does not intersect the boundary of the specified face(s) of $G$. 

\subsection*{Specified Face(s)}

First, we consider the idea of a specified face. Theorem \ref{mainplanar} allows vertices of degree less than $5$ on the boundary of the outer face. However, a graph embedded in a higher genus surface does not have a defined outer face, so the result cannot be directly extended. While the results in this paper deal with planar graphs, and thus could refer to the outer face, to simplify the use of these results in proving the Strong $3$-Flow Conjecture for projective planar graphs or potentially graphs embedded in other surfaces, we define specified faces for all results. \\

Let $G$ be a graph with specified faces $F_G$ and~$F_G^*$. Note that these are not required to both exist. Let $G'$ be a graph obtained from $G$ by one or more operations. Unless otherwise stated, the specified faces $F_{G'}$ and $F_{G'}^*$ are defined as follows:
\begin{enumerate}
\item Suppose that $G'$ is obtained from $G$ by deleting or contracting a connected subgraph of $G$ that has no edge in common with the boundary of $F_G$ or~$F_G^*$. Then $F_{G'}=F_G$ and $F_{G'}^*=F_G^*$. 
\item Suppose that $G'$ is obtained from $G$ by contracting a connected subgraph of $G$ that contains the boundaries of $F_G$ and~$F_G^*$. Then $G'$ has no specified face. Both can be chosen arbitrarily; in general we will chose a specified face incident with the vertex of contraction. 
\item Suppose that $G'$ is obtained from $G$ by deleting an edge $e$ in the boundary of~$F_G$. Let $F$ be the other face incident with~$e$. Then $F_{G'}$ is the face formed by the union of the boundaries of $F$ and $F_G$ (without $e$), and $F_{G'}^*=F_G^*$. Note that if $F=F_G^*$, then $G'$ has only one specified face. The second can be chosen arbitrarily if necessary. 
\item Suppose that $G'$ is obtained from $G$ by deleting a vertex $v$ in the boundary of~$F_G$. Let $F_1$, $F_2$,...,$F_k$ be the other faces incident with~$v$. Then $F_{G'}$ is the face formed by the union of the boundaries of $F_1$, $F_2$,...,$F_k$, and $F_G$ (without the edges incident with $v$), and $F_{G'}^*=F_G^*$. Note that if $F_G^*\in\{F_1,F_2,...,F_k\}$, then $G'$ has only one specified face. The second can be chosen arbitrarily if necessary. 
\item Suppose that $G'$ is obtained from $G$ by contracting a connected subgraph $H$ of $G$ whose intersection with $F_G$ is a path $P$ of length at least one. Then $F_{G'}$ is the face formed by the boundary of $F_G$ without $P$, and $F_{G'}^*=F_G^*$. If the intersection of $H$ with $F_G$ consists of more than one path, this contraction can simply be completed in multiple steps. 
\item Suppose that $G'$ is obtained from $G$ by lifting a pair of adjacent edges $e_1$, $e_2$, where $e_1$ is in the boundary of~$F_G$, $e_2$ is not, and $e_1$ and $e_2$ are consecutive at their common vertex. Let $F_1$ be the other face incident with~$e_1$. Note that $F_1$ is incident with~$e_2$. Let $F_2$ be the other face incident with~$e_2$. Then $F_{G'}$ is the face formed by the union of the boundaries of $F_G$ and $F_2$ (using the lifted edge instead of $e_1$ and $e_2$), and $F_{G'}^*=F_G^*$. Note that if $F_1=F_G^*$, then $F_{G'}^*$ will use the lifted edge instead of $e_1$ and $e_2$, and if $F_2=F_G^*$, then $G'$ has only one specified face. The second can be chosen arbitrarily if necessary. 
\end{enumerate}
When performing these operations we will not explicitly state the new specified faces unless necessary; for example, if we must define a second specified face. 

\subsection*{Edge-Disjoint Paths to the Boundary}

As in \citet{richteretal16}, throughout this paper we are working with graphs for which all vertices not on the boundary of the specified face(s) have at least $5$ edge-disjoint paths to the boundary of the specified face(s). We describe here how reductions to the graph affect this condition. The proof of this result is straightforward, and can be found in \citet{thesis}.

\begin{lemma}
Let $G$ be a graph with specified face $F_G$ such that all vertices not on the boundary of $F_G$ have $5$ edge-disjoint paths to the boundary of~$F_G$. Let $G'$ be a graph obtained from $G$ by 
\begin{enumerate}
\item contracting a subgraph $X$ of $G$ that does not intersect the boundary of $F_G$ to a vertex~$x$,
\item deleting a boundary edge $e$ of~$F_G$,
\item deleting a boundary vertex $x$ of~$F_G$,
\item lifting a pair of adjacent edges $e_1$, $e_2$, where $e_1$ is in the boundary of $F_G$, $e_2$ is not, and $e_1$ and $e_2$ are consecutive at their common vertex, or
\item contracting a subgraph $X$ of $G$ whose intersection with $F_G$ is a path~$P$.
\end{enumerate}
Then all vertices not on the boundary of $F_{G'}$ have $5$ edge-disjoint paths to the boundary of~$F_{G'}$. 
\label{edgedisjoint}
\end{lemma}

We therefore only discuss the preservation of this property in cases where Lemma \ref{edgedisjoint} does not apply. In Section \ref{twoface} we consider graphs with two specified faces, having a vertex in common. Lemma \ref{edgedisjoint} applies analogously to the union of these specified faces. 

\subsection*{Minimal Cuts}
Let $G$ be a graph with a directed vertex $d$. An edge-cut $\delta(A)$ in $G$ with $d\in A$ is \emph{$k$-robust} if $|A|\geq 2$ and $|G-A|\geq k$. \\

Throughout this paper we will perform local reductions on graphs. Many of these reductions will involve considering $2$-robust edge-cuts, either because $G$ has a small edge-cut that must be reduced, or because we must verify that the graph resulting from a reduction does not have any small edge-cuts. In all cases, we first consider the smallest possible edge-cuts. Thus, if we consider a $2$-robust $k$-edge-cut $\delta(A)$ in a graph $G$, we may assume that $G$ has no $2$-robust at most $(k-1)$-edge-cut. It may be assumed that either $G[A]$ is connected, or it consists of two isolated vertices whose degrees sum to~$|\delta(A)|$. The same is true of~$G-A$. Given the size of the cuts we consider, generally both $G[A]$ and $G-A$ will be connected. 

\subsection*{Non-Crossing 3-Edge-Cuts}

Let $\delta(A)$ and $\delta(B)$ be distinct edge-cuts in $G$. We say that $\delta(A)$ and $\delta(B)$ \emph{cross} if $A\cap B$, $A\setminus B$, $B\setminus A$, and $\overline{A\cup B}$ are all non-empty. Throughout this paper we consider graphs that are allowed to have non-crossing $2$-robust $3$-edge-cuts under certain restrictions. The following well-known result allows us to assume that such cuts are non-crossing. We only need the case $k=3$.

\begin{lemma}
Let $k$ be an odd positive integer. If $G$ is a $k$-edge-connected graph, then any two $k$-edge-cuts do not cross. 
\end{lemma}

\section[Increasing the Number of Degree $3$ Vertices]{Increasing the Number of Degree $\mathbf{3}$ Vertices}
\label{planar}

In this section we extend the result in Theorem \ref{mainplanar} to allow a directed vertex $d$ and two other vertices of degree~$3$. An analogous result with three degree $3$ vertices and no directed vertex follows as an immediate corollary. \\

\begin{definition}A \emph{DTS graph} is a graph $G$ embedded in the plane, together with a valid $\mathbb{Z}_3$-prescription function $p: V(G)\rightarrow\{-1,0,1\}$, such that:
\begin{enumerate}
\item $G$ is $3$-edge-connected,
\item $G$ has a specified face $F_G$, and at most three specified vertices $d$, $t$, and~$s$,
\item if $d$ exists, then it has degree $3$, $4$, or $5$, is oriented, and is on the boundary of~$F_G$,
\item if $t$ or $s$ exists, then it has degree $3$ and is on the boundary of~$F_G$,
\item $d$ has degree at most $5-a$ where $a$ is the number of unoriented degree $3$ vertices in~$G$,
\item $G$ has at most three $3$-edge-cuts, which can only be $\delta(d)$, $\delta(t)$, and $\delta(s)$, and
\item every vertex not in the boundary of $F_G$ has $5$ edge-disjoint paths to the boundary of~$F_G$.
\end{enumerate}
We define all $3$-edge-connected graphs on at most two vertices to be DTS graphs, regardless of vertex degrees. \\

A \emph{3DTS graph} is a graph $G$ with the above definition, where (6) is replaced by 
\begin{enumerate}
\item[6'.] all vertices other than $d$, $t$, and $s$ have degree at least $4$, and if $d$, $t$, and $s$ all exist, then every $3$-edge-cut in $G$ separates one of $d$, $t$, and $s$ from the other two.
\end{enumerate}
\end{definition}
We note that the $3$-edge-cuts allowed by (6') are not internal. \\

Our main result for this section is Theorem \ref{maindts}. 
\begin{theorem}
Every DTS graph has a valid orientation. 
\label{maindts}
\end{theorem}
\begin{proof}

Let $G$ be a minimal counterexample with respect to the lexicographic ordering of the pairs $(|E(G)|,|E(G)|-deg(d))$. If $|E(G)|=0$, then $G$ consists of only an isolated vertex, and thus has a trivial valid orientation. If $|E(G)|-deg(d)=0$, then $G$ has an existing valid orientation. Thus we may assume $G$ has at least one unoriented edge. \\

We will establish the following series of properties of $G$. 
\begin{description}
\item[DTS1:] The graph $G$ does not contain a loop, unoriented parallel edges, or a cut vertex.
\item[DTS2:] The graph $G$ does not contain
\begin{enumerate}
\item[a)] a $2$-robust $4$-edge-cut, $\delta_G(A)$, where $d\in A$ and $G-A$ contains at most one of $s$ and~$t$,
\item[b)] a $2$-robust $5$-edge-cut, $\delta_G(A)$, where $d\in A$ and $G-A$ contains neither $s$ nor $t$, or
\item[c)] an internal $2$-robust $6$-edge-cut.
\end{enumerate}
\item[DTS3:] If $e=uv$ is a chord of $F_G$ incident with a vertex $u$ of degree at most $4$, then $deg(u)=4$, $e$ separates $d$ from both $s$ and $t$, and $u$ is incident with $e$ and one other edge in the side containing $d$, while on the side containing $s$ and $t$, $u$ is incident with $e$ and two other edges.
\item[DTS4:] Vertices $d$, $s$, and $t$ exist in $G$.
\item[DTS5:] Vertices $s$ and $t$ are not adjacent.
\end{description}
Let $u$ and $v$ be the boundary vertices adjacent to $t$, and let $w$ be the remaining vertex adjacent to~$t$. Let $x$ and $y$ be the boundary vertices adjacent to $s$, and let $z$ be the remaining vertex adjacent to~$s$.
\begin{description}
\item[DTS6:] Vertices $u$, $v$, $x$, and $y$ have degree $4$.
\item[DTS7:] Edges $uw$, $vw$, $xz$, and $yz$ exist, and $w$ and $z$ have degree $5$.
\item[DTS8:] The vertices $d$, $s$, $t$, $u$, $v$, $x$, and $y$ form the boundary of $F_G$, where either $v=x$ or $u=z$ (up to renaming).
\end{description}

Verifying these properties forms the bulk of the proof of Theorem \ref{maindts}. Property \ref{basicdts} is straightforward, and the proof is omitted. It can be found in \citep{thesis}.\\

\begin{Prop}{DTS1}
The graph $G$ does not contain a loop, unoriented parallel edges, or a cut vertex. 
\label{basicdts}
\end{Prop}
%

Since $G$ has no cut vertex, every face in $G$ is bounded by a cycle. We now consider the presence of small $2$-robust edge-cuts in~$G$. Recall that by definition, $G$ has no $2$-robust at most $3$-edge-cut. Property \ref{cutsdts} categorises three cases that can always be reduced. This is not a comprehensive list, but other small cuts arise arise in special circumstances and will be reduced on a case by case basis. 

\begin{Prop}{DTS2}
The graph $G$ does not contain
\begin{enumerate}
\item[a)] a $2$-robust $4$-edge-cut, $\delta_G(A)$, where $d\in A$ and $G-A$ contains at most one of $s$ and~$t$,
\item[b)] a $2$-robust $5$-edge-cut, $\delta_G(A)$, where $d\in A$ and $G-A$ contains neither $s$ nor $t$, or
\item[c)] an internal $2$-robust $6$-edge-cut.
\end{enumerate}
\label{cutsdts}
\end{Prop}
\begin{proof}\mbox{}
\begin{enumerate}
\item[a)] Suppose that $G$ does contain a $2$-robust $4$-edge-cut, $\delta_G(A)$, with $A$ chosen so that $d\in A$ and $G-A$ contains at most one of $s$ and~$t$. Let $G'$ be the graph obtained from $G$ by contracting $G-A$ to a single vertex. The resulting vertex $v$ has degree~$4$. Since every vertex not incident with $F_G$ has $5$ edge-disjoint paths to the boundary of $F_G$, $F_G$ is incident with edges in $\delta_{G}(A)$. Therefore, $v$ is incident with~$F_{G'}$. If $G'$ contains a cut $\delta_{G'}(B)$ of size at most $3$, then such a cut also exists in $G$, a contradiction unless it is one of the specified vertices. Hence $G'$ is a DTS graph and has a valid orientation by the minimality of~$G$. Transfer this orientation to~$G$. \\

Let $G''$ be the graph obtained from $G$ by contracting $A$ to a single vertex~$d'$. This vertex has degree $4$ and is oriented. Since $d\in V(A)$, $G''$ has only one oriented vertex, which is~$d'$. Since $F_G$ is incident with edges in $\delta_G(A)$, $d'$ is incident with~$F_{G''}$. If $G''$ has a cut $\delta_{G''}(B)$ of size at most $3$, then such a cut also exists in $G$, a contradiction unless it is one of the specified vertices. Thus $G''$ is a DTS graph and has a valid orientation by the minimality of~$G$. Transfer this orientation to $G$ to obtain a valid orientation of $G$, a contradiction. Hence any $4$-edge-cut in $G$ separates $d$ from both $s$ and~$t$.  

\item[b)] This case works in the same way as~a). In $G''$, there is a degree $5$ oriented vertex, and no degree $3$ vertices. Therefore, any $5$-edge-cut in $G$ separates $d$ from an unoriented degree $3$ vertex. 

\item[c)] Contract $G-A$ to a vertex $v$, calling the resulting graph~$G'$. As in a), it is clear that $G'$ is a DTS graph. Therefore, by the minimality of $G$, $G'$ has a valid orientation. Transfer this orientation to~$G$.\\ 

Contract $A$ to a vertex $d'$, delete a boundary edge $e$ incident with $d'$, and call the resulting graph~$G''$. Since the boundary of $F_G$ is contained in $A$, the endpoint of $e$ in $G-A$ has degree at least $5$ in $G$, and therefore has degree at least $4$ in~$G''$. If $G''$ contains an edge-cut $\delta_{G''}(B)$, then either $\delta_G(B)$ or $\delta_G(G-B)$ is a cut in $G$ of size at most one greater. If $G''$ contains a $2$-robust edge cut $\delta_{G''}(B)$ of size three or less, then $G$ contains an analogous internal edge-cut of size four or less. Such a cut does not exist by definition. Thus $G''$ contains no such at most $3$-edge-cut, and is a DTS graph. Therefore, by the minimality of $G$, $G''$ has a valid orientation. Transfer this orientation to $G$ to obtain a valid orientation of $G$, a contradiction. Hence any $6$-edge-cut in $G$ contains edges in the boundary of~$F_G$. \qedhere
\end{enumerate}
\end{proof}\medskip

In the absence of structures such as loops, parallel edges, cut vertices and small edge-cuts, we reduce at low degree vertices in~$G$. It is useful to consider the possible adjacencies of such vertices. We say that a chord $e=\{u,v\}$ of $F_G$ \emph{separates} vertices $x$ and $y$ if $x$ and $y$ are in different components of $F_G-\{u,v\}$. 

\begin{Prop}{DTS3}
If $e=uv$ is a chord of $F_G$ incident with a vertex $u$ of degree at most $4$, then 
\begin{enumerate}
\item[a)] $deg(u)=4$, 
\item[b)] $e$ separates $d$ from both $s$ and $t$, 
\item[c)] $u$ is incident with $e$ and one other edge in the side containing $d$, and 
\item[d)] $u$ is incident with $e$ and two other edges on the side containing $s$ and $t$. 
\end{enumerate}
\label{chordsdts}
\end{Prop}
\begin{proof}
Let $H$ and $K$ be subgraphs of $G$ such that $H\cap K=\{\{u,v\},\{uv\}\}$, $H\cup K=G$, and $d$, if it exists, is in~$H$.\\

Suppose that $\delta(H)$ is not $2$-robust. Then $|V(K)|=3$, and $K$ contains $d$, else $G$ has unoriented parallel edges, contradicting \ref{basicdts}. By definition, either $u$ or $v$ is~$d$. Since $G$ does not contain unoriented parallel edges $deg_H(d)=2$, and $|\delta((H-\{u,v\})\cup\{d\})|=3$, a contradiction. Hence we may assume that $\delta(H)$ is $2$-robust.\\

Suppose that $\delta(K)$ is not $2$-robust. Then $|V(H)|=3$. If $u$ or $v$ is $d$, the same argument applies. Thus we may assume that $d$ is in $V(H)-\{u,v\}$. If there are parallel edges with endpoints $d$ and $u$, then $\delta(\{d,u\})$ is an at most $5$-edge-cut. Orient $u$ and contract the parallel edges between $d$ and $u$, calling the resulting graph~$G'$. Note that the vertex of contraction has the same degree as~$d$. Hence it is clear that $G'$ is a DTS graph, and thus has a valid orientation by the minimality of~$G$. This leads to a valid orientation of $G$, a contradiction. \\

Thus we may assume that there are not parallel edges with endpoints $d$ and~$u$. Then there are parallel edges with endpoints $d$ and~$v$. Since $|\delta(\{d,v\})|\geq 4$, $deg_K(v)\geq 3$. If $deg(u)=3$, then $deg_K(v)\geq 4$, else $|\delta(H)|=3$, a contradiction. Orient $u$ and add a directed edge from $u$ to $v$ in~$K$. Then $K$ is a DTS graph, and has a valid orientation by the minimality of~$G$. This leads to a valid orientation of $G$, a contradiction. Thus $deg(u)=4$. If $K-V(H)$ does not contain both $s$ and $t$, the same argument applies. Thus $s,t\in K-V(H)$, and the chord separates $d$ from $s$ and~$t$. Since there are not parallel edges with endpoints $u$ and $d$, $deg_K(u)=3$. We may now assume that $\delta(K)$ is $2$-robust. \\

By definition, $\delta(H)$ and $\delta(K)$ have size at least~$4$. Hence $deg_G(v)\geq 6$, and so $v$ is not $d$, $s$, or~$t$. In addition, $v$ has degree at least $3$ in both $H$ and~$K$. \\

Suppose that $u\neq d$. Then in $H$, contract~$uv$. The graph $H/uv$ is a DTS graph, and so by the minimality of $G$, $H/uv$ has a valid orientation. Transfer this orientation to $G$, and orient~$u$. In $K$, if $u$ has degree $2$, add an edge $e'$ directed from $u$ to $v$ (in the boundary of~$F_K$). Suppose that $v$ has degree $3$ in $K$ and $K$ does not contain both $s$ and $t$. Since $deg_G(u)\in\{3,4\}$ it is clear that $u$, $v$ are incident with $F_K$, and thus $K+e'$ is a DTS graph. By the minimality of $G$, $K+e'$ has a valid orientation. This leads to a valid orientation of $G$, a contradiction. The remaining case is the one where $v$ has degree $3$ in $K$, and $K$ contains both $s$ and $t$. Then $e$ separates $d$ from $s$ and $t$, as required. \\

Thus $u=d$. Then in both $H$ and $K$ if $deg(u)=2$, add a directed edge $e'$ from $u$ to $v$ (in the boundary of $F_H$ and~$F_K$). If $K$ contains both $s$ and $t$, then $deg(u)=3$, so in $K$, $deg(v)\geq 4$. It is clear that $H+e'$ and $K+e'$ are DTS graphs, so by the minimality of $G$, they have valid orientations. Transfer the orientations of $H+e'$ and $K+e'$ to $G$ to obtain a valid orientation of $G$, a contradiction. Thus no such chord exists. 
\end{proof}

Since we reduce at low degree vertices in $G$, we are especially interested in the specified vertices $d$, $s$, and~$t$. We must first establish the existence of these vertices. The following two claims are needed in order to prove \ref{existsdts}. 

\begin{claim}
If $d$ exists, then it is not adjacent to a vertex of degree at most $4$ via parallel edges.
\label{dadj}
\end{claim}
\begin{proof}
Suppose that $d$ is adjacent to $t$ via parallel edges. Then $|\delta(\{d,t\})|=deg(d)-1\leq 3$ (since $t$ exists, $deg(d)\leq 4$). Since $|\delta(\{d,t\})|\geq 3$, equality holds. Thus $deg(d)=4$ and $s$ does not exist. It follows that $\delta(\{d,t\})$ is a $2$-robust $3$-edge-cut, a contradiction. The same is true of $d$ and $s$.\\

Suppose that $d$ is adjacent to a vertex $v$ of degree $4$ via parallel edges. If they are adjacent via at least $3$ parallel edges, then either $deg(d)=4$ and $\delta(\{d,v\})$ is an at most $2$-edge-cut, or $deg(d)=5$ and $\delta(\{d,v\})$ is an at most $3$-edge-cut, a contradiction. If $s$ and $t$ exist, then $deg(d)=3$, and $\delta(\{d,v\})$ is a $2$-robust at most $3$-edge-cut, a contradiction. Otherwise, orient $v$ and contract $\{d,v\}$, calling the resulting graph $G'$. Then $G'$ has a directed vertex of degree at most $deg_G(d)$, and thus is a DTS graph. By the minimality of $G$, $G'$ has a valid orientation. This leads to a valid orientation of $G$, a contradiction. 
\end{proof}

\begin{claim}
If $d$ exists, then it is not adjacent to a vertex of degree at most $3$.
\label{nonadj}
\end{claim}
\begin{proof}
Suppose that $d$ is adjacent to $t$. Orient $t$ and contract $\{d,t\}$ calling the resulting graph~$G'$. Then $G'$ is a DTS graph and has a valid orientation by the minimality of~$G$. This leads to a valid orientation of $G$, a contradiction.\end{proof}

\begin{Prop}{DTS4} 
Vertices $d$, $s$, and $t$ exist in $G$. 
\label{existsdts}
\end{Prop}
\begin{proof}
Suppose that $d$ does not exist. Let $v$ be a vertex in the boundary of $F_G$ of minimum degree. If $deg(v)\leq 5$, orient $v$, calling the resulting graph~$G'$. Then $G'$ is a DTS graph, and has a valid orientation by the minimality of~$G$. This is a valid orientation of $G$, a contradiction. If $deg(v)\geq 6$, orient and delete a boundary edge incident with $v$, calling the resulting graph~$G'$. If $G'$ has a $2$-robust at most $3$-edge-cut, then $G$ has a $2$-robust at most $4$-edge-cut, a contradiction. Thus $G'$ is a DTS graph and has a valid orientation by the minimality of $G$. This leads to a valid orientation of $G$, a contradiction. Hence $d$ exists. \\

Suppose that $G$ has at most one unoriented vertex ($t$) of degree~$3$. If $d$ has degree $4$ or $5$, let $v$ be a vertex adjacent to $d$ on the boundary of $F_G$ that has degree at least $4$ (such a vertex exists by Claim~\ref{nonadj}). Let $G'$ be the graph obtained by deleting~$dv$. If $G'$ has a $2$-robust at most $3$-edge-cut, then $G$ has a corresponding at most $4$-edge-cut, a contradiction. Thus $G'$ is a DTS graph and has a valid orientation by the minimality of~$G$. This leads to a valid orientation of $G$, a contradiction. \\

Thus $d$ has degree~$3$. By Claim \ref{nonadj}, $d$ is not adjacent to $t$. Delete $d$, calling the resulting graph $G'$. By Claim \ref{dadj}, $G'$ does not have a vertex of degree at most $2$. If $G'$ has a $2$-robust at most $3$-edge-cut, then $G$ has a corresponding at most $4$-edge-cut, a contradiction. If $G'$ contains at most one vertex of degree $3$, then $G'$ is a DTS graph and has a valid orientation by the minimality of~$G$.  Otherwise $G'$ contains two vertices of degree $3$. Orient one, calling the resulting graph $G''$. Then $G''$ is a DTS graph and has a valid orientation by the minimality of~$G$. All cases lead to a valid orientation of $G$, a contradiction. Hence $s$ and $t$ exist. 
\end{proof}

Since $s$ and $t$ exist, $deg(d)=3$. In the following reductions, the graphs produced may be 3DTS graphs. We show that such graphs have a valid orientation. 

\begin{claim}
 \label{3cutsdts}If $G'$ is a 3DTS graph with $|E(G')|<|E(G)|$, then $G'$ has a valid orientation.
\end{claim}
\begin{proof}
Let $G'$ be a minimal counterexample with respect to~$|E(G')|$. Suppose that $G'$ has no $2$-robust $3$-edge-cut. Then $G'$ is a DTS graph and has a valid orientation by the minimality of $G$, a contradiction. Thus we may assume that $G'$ has a $2$-robust $3$-edge-cut $\delta_{G'}(A)$, where $A$ is chosen so that $d\in A$. We may assume that $t\not\in A$. \\

Let $G_1$ be the graph obtained from $G'$ by contracting~$G-A$. Then $G_1$ is a 3DTS graph and has a valid orientation by the minimality of~$G'$. Transfer this orientation to~$G'$. Let $G_2$ be the graph obtained from $G'$ by contracting $A$ to a directed vertex~$d'$. Then $G_2$ is a 3DTS graph and has a valid orientation by the minimality of~$G'$. This leads to a valid orientation of $G'$, a contradiction. 
\end{proof}

\begin{Prop}{DTS5}
Vertices $s$ and $t$ are not adjacent. 
\label{stnotadj}
\end{Prop}
\begin{proof}
Suppose for a contradiction that $s$ and $t$ are adjacent. Let $u$ and $v$ be the boundary neighbours of $s$ and $t$ respectively. We prove the following claims:
\begin{enumerate}
\item[a.] Vertices $s$ and $t$ have a common internal neighbour $w$ of degree $5$, $deg(u)=4$, and $deg(v)=4$.
\item[b.] Vertices $u$ and $w$ are adjacent.
\end{enumerate}

\begin{Inclaim}{DTS5a}
Vertices $s$ and $t$ have a common internal neighbour $w$ of degree $5$, $deg(u)=4$, and $deg(v)=4$.
\end{Inclaim}
\begin{proof}
Assume that either $s$ and $t$ do not have a common internal neighbour of degree~$5$, or at least one of the vertices $u$ and $v$ has degree at least $5$. Let $G'$ be the graph obtained from $G$ by orienting and deleting $s$ and~$t$. Then $G'$ has at most two vertices of degree $3$: two of $u$, $v$, and a possible common neighbour of $s$ and $t$. \\

We now check that $G'$ has no $2$-robust $2$- or $3$-edge-cuts. The cases are indicated in italics. We follow a similar process in most future cases. \emph{Suppose that $G'$ contains a $2$-robust at most $2$-edge-cut $\delta_{G'}(A)$ where $u$ and $v$ are in~$A$.} By definition this cut does not exist in $G$, and so we may assume that $\delta_G(G'-A)$ is an internal cut. We make this assumption without mention in future cases. Then $\delta_G(A\cup\{s,t\})$ is a $2$-robust at most $4$-edge-cut that does not separate $d$ from $s$ or $t$, a contradiction. \emph{Suppose that $G'$ contains a $2$-robust at most $2$-edge-cut $\delta_{G'}(A)$ where $u\in A$ and $v\not\in A$.} Then $\delta_G(A\cup \{s\})$ is a $2$-robust at most $4$-edge-cut that separates $s$ from $t$, a contradiction. \\

\emph{Suppose that $G'$ contains a $2$-robust $3$-edge-cut $\delta_{G'}(A)$ where $u$ and $v$ are in~$A$.} Then $\delta_G(A\cup\{s,t\})$ is a $2$-robust at most $5$-edge-cut that does not separate $d$ from $s$ or $t$, a contradiction. \emph{Suppose that $G'$ contains a $2$-robust $3$-edge-cut $\delta_{G'}(A)$ where $u\in A$ and $v\not\in A$.} Then all such cuts separate $u$ and $v$, so by Claim \ref{3cutsdts}, if $s$ and $t$ do not have a common neighbour of degree $5$, then $G'$ has a valid orientation. Hence we may assume that $s$ and $t$ have a common internal neighbour $w$ of degree $5$. Note that in $G'$, $w$ has degree~$3$. Suppose that $d\in A$. By Claim \ref{3cutsdts}, $w\in A$, and $u$ is the other vertex of degree $3$ in $G'$. Then $\delta_G(G'-A)$ is a $2$-robust $4$-edge-cut in $G$ that does not separate $d$ from $s$ or $t$, a contradiction. Figure \ref{dtsone} shows an analysis of these cuts.\\

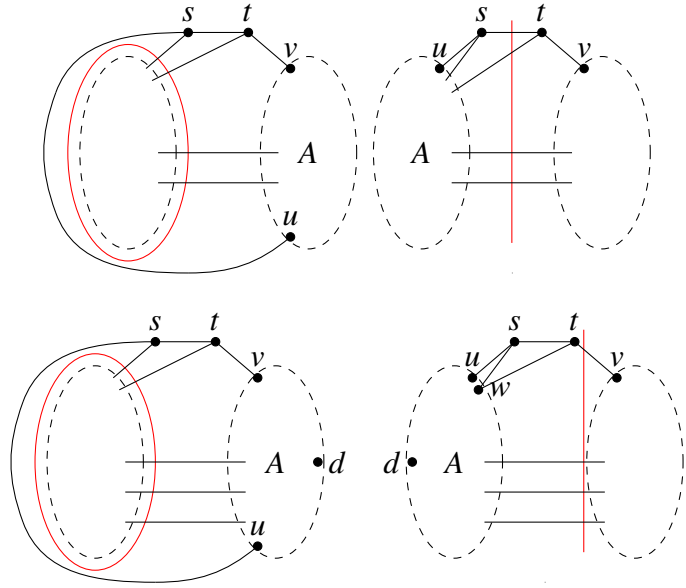
\begin{figure}
\begin{center}
\subfloat{
\begin{tikzpicture}[scale=0.8]
\draw[dashed] (1.5,0) ellipse (0.8cm and 1.6cm) node {$A$};
\draw[dashed] (-1.5,0) ellipse (0.8cm and 1.6cm);
\draw[red] (-1.5,0) ellipse (1cm and 1.8cm);
\filldraw (0.5,2) circle (2pt) node[above] {$t$};
\filldraw (-0.5,2) circle (2pt) node[above] {$s$};
\draw (-1,0) -- (1,0);
\draw (-1,-0.5) -- (1,-0.5);
\draw (0.5,2) -- (1.2,1.4);
\filldraw (1.2,1.4) circle (2pt) node[above] {$v$};
\draw (-0.5,2) -- (0.5,2);
\draw (-0.5,2) -- (-1.2,1.4);
\draw (0.5,2) -- (-1.1,1.2);
\filldraw (1.2,-1.4) circle (2pt) node[above] {$u$};
\draw (-0.5,2) to[out=180,in=70] (-2.7,1);
\draw (-2.7,1) to[out=-110,in=110] (-2.7,-1);
\draw (-0.5,-2) to[out=-180,in=-70] (-2.7,-1);
\draw (-0.5,-2) to[out=0,in=-140] (1.2,-1.4);
\end{tikzpicture}
}
\subfloat{
\begin{tikzpicture}[scale=0.8]
\draw[dashed] (1.5,0) ellipse (0.8cm and 1.6cm);
\draw[dashed] (-1.5,0) ellipse (0.8cm and 1.6cm) node {$A$};
\draw[red] (0,2.2) -- (0,-1.5);
\filldraw (0,-2) circle (0pt);
\filldraw (0.5,2) circle (2pt) node[above] {$t$};
\filldraw (-0.5,2) circle (2pt) node[above] {$s$};
\draw (-1,0) -- (1,0);
\draw (-1,-0.5) -- (1,-0.5);
\draw (0.5,2) -- (1.2,1.4);
\filldraw (1.2,1.4) circle (2pt) node[above] {$v$};
\draw (-0.5,2) -- (0.5,2);
\draw (-0.5,2) -- (-1.2,1.4);
\draw (-0.5,2) -- (-1.1,1.2);
\draw (0.5,2) -- (-1,1);
\filldraw (-1.2,1.4) circle (2pt) node[above] {$u$};
\end{tikzpicture}
}

\subfloat{
\begin{tikzpicture}[scale=0.8]
\draw[dashed] (1.5,0) ellipse (0.8cm and 1.6cm) node {$A$};
\draw[dashed] (-1.5,0) ellipse (0.8cm and 1.6cm);
\draw[red] (-1.5,0) ellipse (1cm and 1.8cm);
\filldraw (0.5,2) circle (2pt) node[above] {$t$};
\filldraw (-0.5,2) circle (2pt) node[above] {$s$};
\draw (-1,0) -- (1,0);
\draw (-1,-0.5) -- (1,-0.5);
\draw (-1,-1) -- (1,-1);
\draw (0.5,2) -- (1.2,1.4);
\filldraw (1.2,1.4) circle (2pt) node[above] {$v$};
\draw (-0.5,2) -- (0.5,2);
\draw (-0.5,2) -- (-1.2,1.4);
\draw (0.5,2) -- (-1.1,1.2);
\filldraw (1.2,-1.4) circle (2pt) node[above] {$u$};
\draw (-0.5,2) to[out=180,in=70] (-2.7,1);
\draw (-2.7,1) to[out=-110,in=110] (-2.7,-1);
\draw (-0.5,-2) to[out=-180,in=-70] (-2.7,-1);
\draw (-0.5,-2) to[out=0,in=-140] (1.2,-1.4);

\filldraw (2.2,0) circle (2pt) node[right] {$d$};
\end{tikzpicture}
}
\subfloat{
\begin{tikzpicture}[scale=0.8]
\draw[dashed] (1.5,0) ellipse (0.8cm and 1.6cm);
\draw[dashed] (-1.5,0) ellipse (0.8cm and 1.6cm) node {$A$};
\filldraw (0.5,2) circle (2pt) node[above] {$t$};
\filldraw (-0.5,2) circle (2pt) node[above] {$s$};
\draw[red] (0.65,2.2) -- (0.65,-1.5);
\draw (-0.5,2) -- (0.5,2);
\draw (-0.5,2) -- (-1.2,1.4);
\draw (0.5,2) -- (1.2,1.4);
\draw (-0.5,2) -- (-1.1,1.2);
\draw (0.5,2) -- (-1.1,1.2);
\filldraw (-1.1,1.2) circle (2pt) node[right] {$w$};
\filldraw (0,-2) circle (0pt);
\draw (-1,0) -- (1,0);
\draw (-1,-0.5) -- (1,-0.5);
\draw (-1,-1) -- (1,-1);
\filldraw (1.2,1.4) circle (2pt) node[above] {$v$};
\filldraw (-1.2,1.4) circle (2pt) node[above] {$u$};
\filldraw (-2.2,0) circle (2pt) node[left] {$d$};
\end{tikzpicture}
}
\caption{\ref{stnotadj}: Analysis of cuts (1).}
\label{dtsone}
\end{center}
\end{figure}

We conclude that $G'$ is a DTS graph and has a valid orientation by the minimality of~$G$. This leads to a valid orientation of $G$, a contradiction. 
\end{proof}

Therefore $s$ and $t$ have a common internal neighbour~$w$ of degree $5$, $deg(u)=4$, and $deg(v)=4$. Let the edges incident with $u$ be $e_1$, $e_2$, $e_3$, $e_s$ in order, where $e_1$ is on the boundary of~$F_G$ and $e_s=su$. By \ref{chordsdts}, $e_3$ is not a chord. The same is true at~$v$. 

\begin{Inclaim}{DTS5b}
Vertices $u$ and $w$ are adjacent.
\end{Inclaim}
\begin{proof}
Suppose that $u$ and $w$ are not adjacent. Lift the pair of edges $e_1$, $e_2$, and orient and delete $u$, $s$, and $t$, calling the resulting graph~$G'$. This reduction can be seen in Figure~\ref{dtstwo}. Then $G'$ has at most two unoriented vertices of degree $3$ ($v$ and~$w$). \\

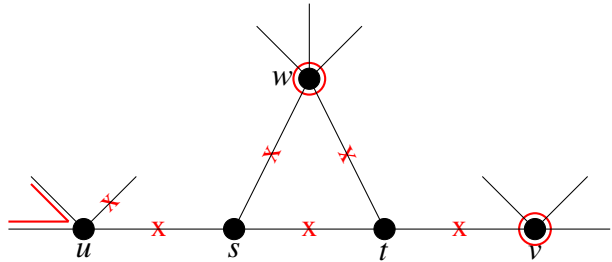
\begin{figure}
\begin{center}
\begin{tikzpicture}[scale = 2]
\filldraw (-0.5,0) circle (2pt) node[below=1pt] {$s$};
\filldraw (0.5,0) circle (2pt) node[below=1pt] {$t$};
\filldraw (-1.5,0) circle (2pt) node[below=1pt] {$u$};
\filldraw (1.5,0) circle (2pt) node[below=2pt] {$v$};
\filldraw (0,1) circle (2pt) node[left=2pt] {$w$};
\draw[red] (0,0) node {x};
\draw[red] (1,0) node {x};
\draw[red] (-1,0) node {x};
\draw[red] (-0.25,0.5) node[rotate=63] {x};
\draw[red] (0.25,0.5) node[rotate=117] {x};
\draw[red] (-1.325,0.175) node[rotate=45] {x};
\draw[red,thick] (-1.6,0.05) -- (-2,0.05);
\draw[red,thick] (-1.6,0.05) -- (-1.85,0.3);
\draw[red,thick] (0,1) circle (3pt);
\draw[red,thick] (1.5,0) circle (3pt);
\draw (-2,0) -- (2,0);
\draw (-0.5,0) -- (0,1);
\draw (0.5,0) -- (0,1);
\draw (-1.5,0) -- (-1.85,0.35);
\draw (-1.5,0) -- (-1.15,0.35);
\draw (1.5,0) -- (1.85,0.35);
\draw (1.5,0) -- (1.15,0.35);
\draw (0,1) -- (0,1.5);
\draw (0,1) -- (0.35,1.35);
\draw (0,1) -- (-0.35,1.35);
\end{tikzpicture}
\caption{\ref{stnotadj}: Reduction when $u$ and $w$ are not adjacent.}
\label{dtstwo}
\end{center}
\end{figure}

\emph{Suppose that $G'$ contains a $2$-robust at most $2$-edge-cut $\delta_{G'}(A)$, where $v$ and the lifted edge are in~$A$.} Then $\delta_G(A\cup\{s,t,u\})$ is an internal $2$-robust at most $5$-edge-cut, a contradiction. \emph{Suppose that $G'$ contains a $2$-robust at most $2$-edge-cut $\delta_{G'}(A)$, where $v\in A$ and the lifted edge is in~$G'-A$.} Then $w\in A$, else $\delta_G(A)$ is a $2$-robust $3$-edge-cut. Similarly, the endpoint of $e_3$ is in $A$, else $\delta_G(A\cup\{s,t\})$ is a $2$-robust $3$-edge-cut. Now $\delta_G(A\cup\{u,s,t\})$ is a $4$-edge-cut. By \ref{cutsdts}, $d\in G'-A$. Contract $A$ in $G$ to form a graph~$\bar{G}$. Then $\bar{G}$ is a DTS graph and has a valid orientation by the minimality of~$G$. Transfer this orientation to $G$, contract $G-A$, and delete $tv$ and $tw$ to form a graph $\bar{G}'$ with an oriented degree $4$ vertex. If $\bar{G}'$ contains an at most $3$-edge-cut, then $G$ contains an at most $5$-edge-cut that does not separate $d$, $s$, and $t$, a contradiction. Hence $\bar{G}'$ is a DTS graph and has a valid orientation by the minimality of~$G$. This leads to a valid orientation of $G$, a contradiction. Therefore $G'$ is $3$-edge-connected. \\

\emph{Suppose that $G'$ contains a $2$-robust $3$-edge-cut $\delta_{G'}(A)$, where $v$ and the lifted edge are in~$A$.} Then $\delta_G(A\cup\{s,t,u\})$ is an internal $2$-robust at most $6$-edge-cut, a contradiction. \emph{Suppose that $G'$ contains a $2$-robust $3$-edge-cut $\delta_{G'}(A)$, where $v\in A$, and the lifted edge is in~$G'-A$.} If all such cuts have the property that $d\in G'-A$, then $G'$ is a 3DTS graph and has a valid orientation by Claim~\ref{3cutsdts}. This leads to a valid orientation of $G$, a contradiction. Hence we may assume that $d\in A$. The endpoint of $e_3$ is in $G'-A$, else $\delta_G(A\cup\{u,s,t\})$ is a $2$-robust $5$-edge-cut that does not separate $d$ from $s$ or $t$, a contradiction. Also, $w\in G'-A$, else $\delta_G(A\cup\{s,t\})$ is a $2$-robust $4$-edge-cut that does not separate $d$ from $s$ or $t$, a contradiction. Then all such $3$-edge-cuts in $G'$ separate $w$ from $v$ and $d$, and so $G'$ has a valid orientation by Claim~\ref{3cutsdts}. Therefore $G'$ has no such $2$-robust at most $3$-edge-cuts. Figure \ref{dtsthree} shows an analysis of these cuts. \\

\begin{figure}
\begin{center}
\subfloat{
\begin{tikzpicture}[scale=0.8]
\draw[dashed] (1.5,0) ellipse (0.8cm and 1.6cm) node {$A$};
\draw[dashed] (-1.5,0) ellipse (0.8cm and 1.6cm);
\draw[red] (-1.5,0) ellipse (1cm and 1.8cm);
\filldraw (0.5,2) circle (2pt) node[above] {$t$};
\filldraw (0,2) circle (2pt) node[above] {$s$};
\filldraw (-0.5,2) circle (2pt) node[above] {$u$};
\filldraw (-1,1) circle (2pt) node[left] {$w$};
\draw (-1,0) -- (1,0);
\draw (-1,-0.5) -- (1,-0.5);
\draw (0.5,2) -- (1.2,1.4);
\filldraw (1.2,1.4) circle (2pt) node[above] {$v$};
\draw (-0.5,2) -- (0.5,2);
\draw (-0.5,2) -- (-1.2,1.4);
\draw (0,2) -- (-1,1);
\draw (0.5,2) -- (-1,1);
\draw (-0.5,2) to[out=180,in=70] (-2.7,1);
\draw (-2.7,1) to[out=-110,in=110] (-2.7,-1);
\draw (-0.5,-2) to[out=-180,in=-70] (-2.7,-1);
\draw (-0.5,-2) to[out=0,in=-140] (1.2,-1.4);

\draw (-0.5,2) to[out=170,in=70] (-2.8,1.1);
\draw (-2.8,1.1) to[out=-110,in=110] (-2.8,-1.1);
\draw (-0.5,-2.1) to[out=-180,in=-70] (-2.8,-1.1);
\draw (-0.5,-2.1) to[out=0,in=-140] (1.3,-1.5);
\end{tikzpicture}
}
\subfloat{
\begin{tikzpicture}[scale=0.8]
\draw[dashed] (1.5,0) ellipse (0.8cm and 1.6cm) node {$A$};
\draw[dashed] (-1.5,0) ellipse (0.8cm and 1.6cm);
\filldraw (0,-2) circle (0pt);
\filldraw (0.5,2) circle (2pt) node[above] {$t$};
\filldraw (0,2) circle (2pt) node[above] {$s$};
\filldraw (-0.5,2) circle (2pt) node[above] {$u$};
\filldraw (1,1) circle (2pt) node[right] {$w$};
\draw (-1,0) -- (1,0);
\draw (-1,-0.5) -- (1,-0.5);
\draw (0.5,2) -- (1.2,1.4);
\filldraw (1.2,1.4) circle (2pt) node[above] {$v$};
\draw (-0.5,2) -- (0.5,2);
\draw (-0.5,2) -- (-1.2,1.4);
\draw (-0.5,2) -- (-1.1,1.2);
\draw (0,2) -- (1,1);
\draw (0.5,2) -- (1,1);
\draw (-0.5,2) -- (0.9, 0.8);
\filldraw (-2.2,0) circle (2pt) node[left] {$d$};
\end{tikzpicture}
}

\subfloat{
\begin{tikzpicture}[scale=0.8]
\draw[dashed] (1.5,0) ellipse (0.8cm and 1.6cm) node {$A$};
\draw[dashed] (-1.5,0) ellipse (0.8cm and 1.6cm);
\draw[red] (-1.5,0) ellipse (1cm and 1.8cm);
\filldraw (0.5,2) circle (2pt) node[above] {$t$};
\filldraw (0,2) circle (2pt) node[above] {$s$};
\filldraw (-0.5,2) circle (2pt) node[above] {$u$};
\filldraw (-1,1) circle (2pt) node[left] {$w$};
\draw (-1,0) -- (1,0);
\draw (-1,-0.5) -- (1,-0.5);
\draw (-1,-1) -- (1,-1);
\draw (0.5,2) -- (1.2,1.4);
\draw (0,2) -- (-1,1);
\draw (0.5,2) -- (-1,1);
\filldraw (1.2,1.4) circle (2pt) node[above] {$v$};
\draw (-0.5,2) -- (0.5,2);
\draw (-0.5,2) -- (-1.2,1.4);
\draw (-0.5,2) to[out=180,in=70] (-2.7,1);
\draw (-2.7,1) to[out=-110,in=110] (-2.7,-1);
\draw (-0.5,-2) to[out=-180,in=-70] (-2.7,-1);
\draw (-0.5,-2) to[out=0,in=-140] (1.2,-1.4);

\draw (-0.5,2) to[out=170,in=70] (-2.8,1.1);
\draw (-2.8,1.1) to[out=-110,in=110] (-2.8,-1.1);
\draw (-0.5,-2.1) to[out=-180,in=-70] (-2.8,-1.1);
\draw (-0.5,-2.1) to[out=0,in=-140] (1.3,-1.5);

\end{tikzpicture}
}
\subfloat{
\begin{tikzpicture}[scale=0.8]
\draw[dashed] (1.5,0) ellipse (0.8cm and 1.6cm) node {$A$};
\draw[dashed] (-1.5,0) ellipse (0.8cm and 1.6cm);
\filldraw (0.5,2) circle (2pt) node[above] {$t$};
\filldraw (0,2) circle (2pt) node[above] {$s$};
\filldraw (-0.5,2) circle (2pt) node[above] {$u$};
\filldraw (-0.9,0.8) circle (2pt) node[left] {$w$};
\draw (-0.5,2) -- (0.5,2);
\draw (-0.5,2) -- (-1.2,1.4);
\draw (0.5,2) -- (1.2,1.4);
\filldraw (0,-2) circle (0pt);
\draw (-1,0) -- (1,0);
\draw (-1,-0.5) -- (1,-0.5);
\draw (-1,-1) -- (1,-1);
\draw (-0.5,2) -- (-1.1,1.2);
\draw (-0.5,2) -- (-1,1);
\draw (0,2) -- (-0.9,0.8);
\draw (0.5,2) -- (-0.9,0.8);
\filldraw (1.2,1.4) circle (2pt) node[above] {$v$};
\filldraw (2.2,0) circle (2pt) node[right] {$d$};
\end{tikzpicture}
}
\caption{\ref{stnotadj}: Analysis of cuts (2).}
\label{dtsthree}
\end{center}
\end{figure}
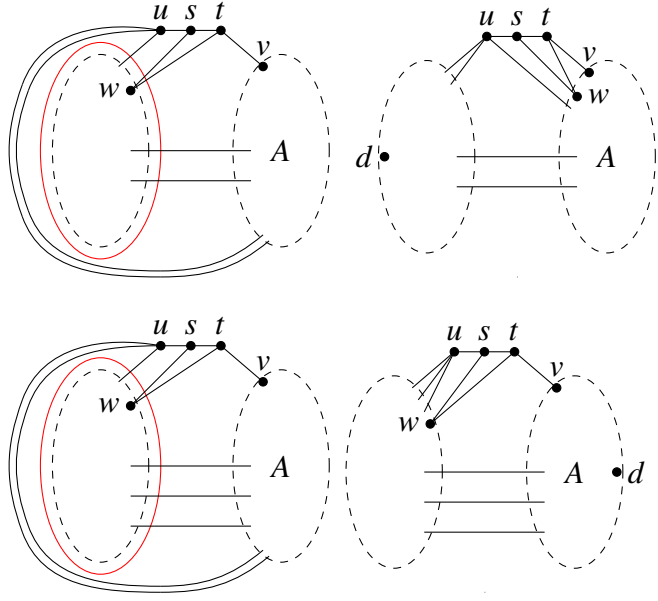

Thus far we have omitted consideration of the case where the lifted edge is an edge in the cut~$\delta(A)$. \emph{Suppose that $G'$ contains a $2$-robust at most $3$-edge-cut $\delta_{G'}(A)$ where $v\in A$ and the lifted edge is an edge of~$\delta(A)$.} If the endpoint of $e_1$ is in $G'-A$, then $\delta_G(G'-A)$ is a $2$-robust at most $3$-edge-cut, a contradiction. If the endpoint of $e_1$ is in $A$, then the argument is analogous to (but simpler than) that for the case where $v$ and the lifted edge are on the same side of the cut. In future reductions we omit discussion of the cases where the lifted edge is an edge of the cut. \\

Hence $G'$ is a DTS graph and has a valid orientation by the minimality of~$G$. This leads to a valid orientation of $G$, a contradiction. 
\end{proof}

By symmetry, $v$ and $w$ are adjacent. This provides sufficient structure to complete the proof. Suppose that $e_2$ and the analogous edge $f_2$ incident with $v$ are both chords. By \ref{chordsdts}, $e_2$ and $f_2$ separate $d$ from $s$ and~$t$. Thus $e_2=f_2$. Then $G$ has a $1$-edge-cut: the remaining edge incident with $w$, a contradiction. Hence $e_2$ and $f_2$ are not both chords. \\

By \ref{cutsdts}, $u$ and $v$ are not both adjacent to~$d$. Without loss of generality, assume that $u$ is not adjacent to~$d$. Suppose that $e_2$ is a chord. Then by \ref{chordsdts}, $e_2$ separates $d$ from $s$ and $t$. Since $e_2$ is not incident with $v$, it also separates $d$ from $v$, so $v$ is not adjacent to~$d$. Therefore at least one of $u$ and $v$ is not adjacent to $d$ and is not incident with a chord. Without loss of generality, we assume this vertex is $u$. Let $z$ be the other endpoint of~$e_1$. \\

Orient and delete $e_1$ and $e_2$ to satisfy $p(u)$ (which lifts $ut$ and $uw$ at $u$), and contract $\{u,s,t,v,w\}$ to a single vertex $c$ of degree $3$, calling the resulting graph~$G'$. Then $G'$ has at most two unoriented degree $3$ vertices ($c$ and~$z$). This reduction is shown in Figure \ref{dtsfour}.\\

\begin{figure}
\begin{center}
\begin{tikzpicture}[scale = 2]
\filldraw (-0.5,0) circle (2pt) node[below=1pt] {$s$};
\filldraw (0.5,0) circle (2pt) node[below=1pt] {$t$};
\filldraw (-1.5,0) circle (2pt) node[below=1pt] {$u$};
\filldraw (1.5,0) circle (2pt) node[below=1pt] {$v$};
\filldraw (0,1) circle (2pt) node[left=1pt] {$w$};
\filldraw (-2.5,0) circle (2pt) node[below=2pt] {$z$};
\draw[red] (-2,0) node {x};
\draw[red] (-1.675,0.175) node[rotate=-45] {x};
\draw[red,thick] (-1.35,0.05) -- (-1,0.05);
\draw[red,thick] (-1.35,0.05) -- (-1.05,0.25);
\draw[red,thick] (-2.5,0) circle (3pt);
\draw (-3,0) -- (2,0);
\draw (-0.5,0) -- (0,1);
\draw (0.5,0) -- (0,1);
\draw (-1.5,0) -- (-1.85,0.35);
\draw (-1.5,0) -- (0,1);
\draw (1.5,0) -- (1.85,0.35);
\draw (1.5,0) -- (0,1);
\draw (0,1) -- (0,1.5);
\draw (-2.5,0) -- (-2.85,0.35);
\draw (-2.5,0) -- (-2.15,0.35);

\draw[red] (-1.65,-0.2) to[out=90,in=180] (0,1.2);
\draw[red] (1.65,-0.2) to[out=90,in=0] (0,1.2);
\end{tikzpicture}
\caption{\ref{stnotadj}: Reduction when $u$ and $w$ are adjacent.}
\label{dtsfour}
\end{center}
\end{figure}
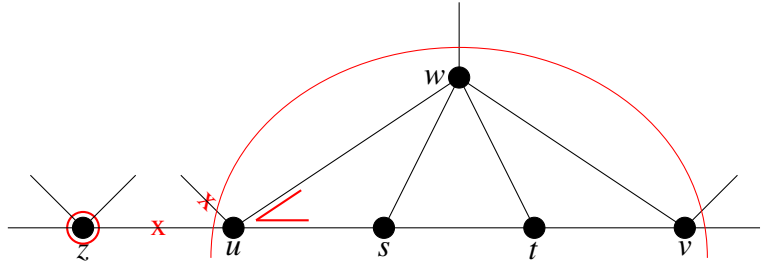

\emph{Suppose that $G'$ contains a $2$-robust at most $3$-edge-cut $\delta_{G'}(A)$ where $z$ and $c$ are in~$A$.} Then $G$ contains a $2$-robust internal at most $4$-edge-cut, a contradiction. \\

\emph{Suppose that $G'$ contains a $2$-robust at most $2$-edge-cut $\delta_{G'}(A)$ where $z\in G'-A$ and $c\in A$.} Since $\delta_G(A)$ and $\delta_G(A\cup\{u\})$ are $4$-edge-cuts, $d\not\in A$. If $v$ is adjacent to $d$ or incident with a chord, then $G$ has an internal at most $3$-edge-cut $\delta_G(A-c)$, a contradiction. Hence $v$ is not adjacent to $d$ or incident with a chord. We consider applying the same reduction at~$v$. If an analogous $4$-edge-cut $\delta_G(B)$ exists using edges incident with $v$, then these cuts cross. \\

We have $\{s,t,u,v,w\}\subseteq A\cap B$ in $G$ (where we replace $c$ with $\{s,t,u,v,w\}$ in $G$), and $d\in G-(A\cup B)$. By construction, $A-B$ and $B-A$ each contain at least two vertices: the neighbours of $u$ and~$v$. Thus \ref{cutsdts} implies that $|\delta_G(A-B)|,|\delta_G(B-A)|\geq 6$. This is not possible given that $|\delta_G(A)|=|\delta_G(B)|=4$. Thus we may apply the same reduction at $v$ without producing a $2$-edge-cut. Hence we may assume that no such cut exists, and $G'$ is $3$-edge-connected.\\

\emph{Suppose that $G'$ contains a $2$-robust $3$-edge-cut $\delta_{G'}(A)$ where $z\in G'-A$ and $c\in A$.} Then $G'$ is a 3DTS graph and has a valid orientation by Claim~\ref{3cutsdts}. This leads to a valid orientation of $G$, a contradiction. Hence no such cut exists. \\

We conclude that $G'$ is a DTS graph, and has a valid orientation by the minimality of~$G$. This leads to a valid orientation of $G$, a contradiction. Hence $s$ and $t$ are not adjacent. 
\end{proof}

Let $u$ and $v$ be the boundary vertices adjacent to $t$, and let $w$ be the remaining vertex adjacent to~$t$. Let $x$ and $y$ be the boundary vertices adjacent to $s$, and let $z$ be the remaining vertex adjacent to~$s$. Note that $w$ and $z$ are internal and have degree at least~$5$, and that none of $u$, $v$, $w$, $x$, $y$, and $z$ are in the set $\{d,s,t\}$. 

\begin{Prop}{DTS6}
Vertices $u$, $v$, $x$, and $y$ have degree $4$. 
\end{Prop}
\begin{proof}
Suppose without loss of generality that $v$ has degree at least~$5$. Let $G'$ be the graph obtained from $G$ by orienting and deleting~$t$. Then $s$ and $u$ are the only possible unoriented degree $3$ vertices in~$G'$. \emph{Suppose that $G'$ contains a $2$-robust at most $2$-edge-cut~$\delta_{G'}(A)$.} Then $\delta_G(A)$ or $\delta_G(A\cup\{t\})$ is a $2$-robust at most $3$-edge-cut, a contradiction. \\

\emph{Suppose that $G'$ contains a $2$-robust $3$-edge-cut~$\delta_{G'}(A)$. }If $u$ and $v$ are both in $A$, then $G$ contains an internal $4$-edge-cut, a contradiction. Let $u\in A$ and $v\in G'-A$. If all such $3$-edge-cuts separate $d$ from $u$, then by Claim \ref{3cutsdts} $G'$ has a valid orientation. This leads to a valid orientation of $G$, a contradiction. Hence we may assume that $d\in A$. If all such cuts separate $s$ from $u$, then similarly, $G'$ has a valid orientation. This leads to a valid orientation of $G$, a contradiction. Hence we may assume that $s\in A$. Then $\delta_G(A)$ or $\delta_G(A\cup\{t\})$ is a $4$-edge-cut that does not separate $s$ from $d$, a contradiction. Hence no such cut exists. Thus $G'$ is a DTS graph and has a valid orientation by the minimality of~$G$. This leads to a valid orientation of $G$, a contradiction. Hence $v$ has degree~$4$. The same argument applies for $u$, $x$, and~$y$. 
\end{proof}

\begin{Prop}{DTS7}
The edges $uw$, $vw$, $xz$, and $yz$ exist, and $w$ and $z$ have degree $5$.
\label{copys}
\end{Prop}
\begin{proof}
Suppose without loss of generality that either $u$ and $w$ are not adjacent or that $deg(w)\geq 6$. Let $e_1$, $e_2$, $e_3$, and $e_t$ be the edges incident with $u$ in order, where $e_1$ is on the boundary of~$F_G$ and $e_t=ut$. By \ref{chordsdts}, $e_3$ is not a chord. Let $G'$ be the graph obtained from $G$ by lifting the pair of edges $e_1$, $e_2$, and orienting and deleting $u$ and~$t$. Then $v$ and $s$ are the only possible unoriented degree $3$ vertices in~$G'$. \\

\emph{Suppose that $G'$ contains a $2$-robust at most $3$-edge-cut $\delta_{G'}(A)$ where $v$ and the lifted edge are in~$A$.} Then $\delta_G(A)$ is a $2$-robust internal at most $5$-edge-cut, a contradiction. \\

\emph{Suppose that $G'$ contains a $2$-robust at most $2$-edge-cut $\delta_{G'}(A)$ where $v\in A$ and the lifted edge is in~$G'-A$.} Since $G$ does not contain a $2$-robust at most $3$-edge-cut, $w$ and the endpoint of $e_3$ are in~$A$. Since $\delta_G(A\cup\{u,t\})$ is a $2$-robust $4$-edge-cut, $d\in G'-A$, and $s\in A$. Figure \ref{dtsfive} shows this graph. In $G$, contract $A$ to form a graph~$\hat{G}$. Then $\hat{G}$ is a DTS graph and has a valid orientation by the minimality of~$G$. Transfer this orientation to~$G$. Contract $G-A$ and delete $tv$ and $tw$ to form a graph~$\hat{G}'$. Then $\hat{G}'$ has an oriented degree $3$ vertex $d'$ and two vertices of degree $3$ ($s$ and~$v$). If $\hat{G}'$ contains a $2$-robust at most $2$-edge-cut, then $G$ contains a $2$-robust at most $3$-edge-cut, a contradiction. If $\hat{G}'$ contains a $2$-robust $3$-edge-cut, it necessarily separates $d'$ from $s$ or $v$ (else $G$ has a $2$-robust at most $5$-edge-cut with $d$, $s$, and $t$ on the same side), so $\hat{G}'$ has a valid orientation by Claim~\ref{3cutsdts}. Thus $\hat{G'}$ is a DTS graph and has a valid orientation by the minimality of~$G$. This leads to a valid orientation of $G$, a contradiction. Hence $G'$ is $3$-edge-connected.\\

\begin{figure}
\begin{center}
\begin{tikzpicture}
\draw[dashed] (1.5,0) ellipse (0.8cm and 1.6cm);
\draw[dashed] (-1.5,0) ellipse (0.8cm and 1.6cm) node {$A$};
\filldraw (0.5,2) circle (2pt) node[above] {$u$};
\filldraw (-0.5,2) circle (2pt) node[above] {$t$};
\draw (-1,0) -- (1,0);
\draw (-1,-0.5) -- (1,-0.5);
\draw (0.5,2) -- (1.2,1.4);
\draw (0.5,2) -- (1.1,1.2);
\filldraw (-1.2,1.4) circle (2pt) node[above] {$v$};
\draw (-0.5,2) -- (0.5,2);
\draw (-0.5,2) -- (-1.2,1.4);
\draw (-0.5,2) -- (-1.1,1.2);
\draw (0.5,2) -- (-1,1);
\filldraw (2.2,0) circle (2pt) node[right] {$d$};
\filldraw (-2.2,0) circle (2pt) node[left] {$s$};
\draw[red] (-0.9,2) to[out=-50,in=90] (0,-1.5);
\end{tikzpicture}
\caption{\ref{copys}: Reduce the $5$-edge-cut $\delta(A)$.}
\label{dtsfive}
\end{center}
\end{figure}
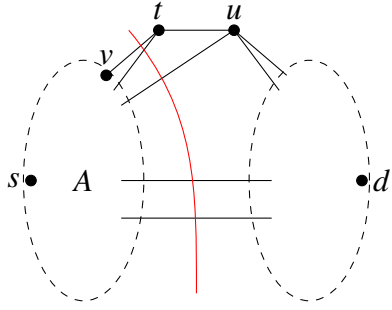

\emph{Suppose that $G'$ contains a $2$-robust at most $3$-edge-cut $\delta_{G'}(A)$ where $v\in A$ and the lifted edge is in~$G'-A$.} If all such $3$-edge-cuts separate $d$ from $v$, then by Claim \ref{3cutsdts} $G'$ has a valid orientation. This leads to a valid orientation of $G$, a contradiction. Hence we may assume that $d\in A$. If $s\in A$, then either $\delta_G(A\cup\{t\})$ or $\delta_G(A\cup\{u,t\})$ is a $2$-robust at most $5$-edge-cut that does not separate $d$ from $s$ or $t$, contradicting \ref{cutsdts}. Hence $s\in G'-A$. Therefore $\delta_{G'}(A)$ separates $s$ from $d$ and~$v$. By Claim \ref{3cutsdts}, $G'$ has a valid orientation. This leads to a valid orientation of $G$, a contradiction. Thus $G'$ has no $2$-robust at most $3$-edge-cut. \\

Therefore $G'$ is a DTS graph and has a valid orientation by the minimality of~$G$. This leads to a valid orientation of $G$, a contradiction. We conclude that $u$ and $w$ are adjacent and $deg(w)=5$. The other cases are equivalent. 
\end{proof}

\begin{Prop}{DTS8}
The vertices $d$, $s$, $t$, $u$, $v$, $x$, and $y$ form the boundary of $F_G$, where either $v=x$ or $u=z$ (up to renaming). 
\label{boundary}
\end{Prop}
\begin{proof}
Suppose without loss of generality that $u$ is not adjacent to $s$ or to~$d$. Suppose that $u$ is incident with a chord. If $u$ and $v$ are both incident with chords, then by \ref{chordsdts} these chords cross, a contradiction. Thus we may assume that $v$ is not incident with a chord. If $v$ is adjacent to $d$, then $u$ is not incident with a chord, by \ref{chordsdts}, a contradiction. If $v$ is adjacent to $s$, then $v=x$. If $u$ and $y$ are both incident with chords, then by \ref{chordsdts} these chords cross, a contradiction. Thus we may assume that $y$ is not incident with a chord. If $y$ is adjacent to $d$, then $u$ is not incident with a chord, by \ref{chordsdts}. Then $y$ is neither incident with a chord nor adjacent to~$d$. Hence there exists a vertex in the set $\{u,v,x,y\}$ that is adjacent to exactly one vertex in $\{s,t,d\}$ and is not incident with a chord. Without loss of generality, assume this vertex is~$u$. Let $e_1$, $e_2$, $e_3$, and $e_t$ be the edges incident with $u$ in order, where $e_1$ is on the boundary of~$F_G$ and $e_t=ut$. Let $q$ be the other endpoint of~$e_1$. \\

Let $G'$ be the graph obtained from $G$ by orienting and deleting $e_1$ and $e_2$ to satisfy $p(u)$, and contracting $\{u,t,v,w\}$ to a single vertex $c$ of degree~$4$. Then $G'$ has at most two vertices of degree $3$ ($s$ and~$q$). \emph{Suppose that $G'$ contains a $2$-robust at most $3$-edge-cut $\delta_{G'}(A)$ where $c$ and $q$ are in~$A$.} Then $\delta_G(G'-A)$ is a $2$-robust internal at most $4$-edge-cut in $G$, a contradiction. \\

\emph{Suppose that $G'$ contains a $2$-robust at most $2$-edge-cut where $c\in A$ and $q\in G'-A$. }Then $s\in A$ and $d\in G'-A$, else $G$ has a $2$-robust at most $4$-edge-cut that does not separate $d$ from $s$ and $t$, a contradiction to \ref{cutsdts}. Figure \ref{dtssix} shows this graph. In $G$, contract $(A-\{c\})\cup\{v,w\}$ to a vertex, calling the resulting graph~$\bar{G}$. Then $\bar{G}$ is a DTS graph and has a valid orientation by the minimality of~$G$. Transfer this orientation to $G$, contract $(G'-A)\cup\{u,t\}$ to a vertex, and delete $tv$ and $tw$, calling the resulting graph~$\bar{G}'$. Then $\bar{G}'$ has a directed degree $3$ vertex $d'$ and two vertices of degree $3$ ($s$ and~$v$). If $\bar{G}'$ contains a $2$-robust at most $2$-edge-cut, then $G$ contains a $2$-robust at most $4$-edge-cut that does not separate $d$ from $s$ and $t$, a contradiction. If $\bar{G}'$ contains a $2$-robust $3$-edge-cut, it necessarily separates $d'$ from $s$ or $v$, so $\bar{G}'$ has a valid orientation by Claim~\ref{3cutsdts}. Thus $\bar{G'}$ is a DTS graph and has a valid orientation by the minimality of~$G$. This leads to a valid orientation of $G$, a contradiction. Hence $G'$ is $3$-edge-connected.\\

\begin{figure}
\begin{center}
\begin{tikzpicture}
\draw[dashed] (1.5,0) ellipse (0.8cm and 1.6cm) node {$A$};
\draw[dashed] (-1.5,0) ellipse (0.8cm and 1.6cm);
\draw (-1,0) -- (1,0);
\draw (-1,-0.5) -- (1,-0.5);
\filldraw (-1.2,1.4) circle (2pt) node[above] {$q$};
\filldraw (1.2,1.4) circle (2pt) node[above] {$u$};
\filldraw (1.8,0.7) circle (2pt) node[above] {$v$};
\filldraw (1.5,1.05) circle (2pt) node[above] {$t$};
\filldraw (1.15,0.75) circle (2pt) node[below] {$w$};
\draw (1.2,1.4) -- (-1.2,1.4);
\draw (1.2,1.4) -- (-1.1,1.2);
\draw (1.2,1.4) -- (1.8,0.7);
\draw (1.5,1.05) -- (1.15,0.75);
\draw (1.2,1.4) -- (1.15,0.75);
\draw (1.8,0.7) -- (1.15,0.75);
\draw (1.8,0.7) -- (1.98,0.49);
\draw (1.8,0.7) -- (1.8,0.45);
\draw (1.15,0.75) -- (0.9,0.75);
\draw (1.15,0.75) -- (0.92,0.65);
\filldraw (2.2,0) circle (2pt) node[right] {$s$};
\filldraw (-2.2,0) circle (2pt) node[left] {$d$};
\draw[red] (1,0.9) to[out=-180,in=90] (0,-1.5);
\draw[red] (1,0.9) to[out=0,in=170] (2.5,0.8);
\end{tikzpicture}
\caption{\ref{boundary}: Reduce the $5$-edge-cut $\delta(A-\{u,t\}))$.}
\label{dtssix}
\end{center}
\end{figure}
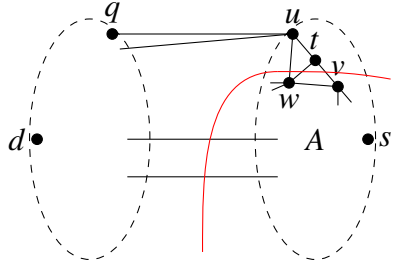

\emph{Suppose that $G'$ contains a $2$-robust $3$-edge-cut $\delta_{G'}(A)$, where $c\in A$ and $q\in G'-A$.} If all such cuts separate $d$ from $q$, then $G'$ has a valid orientation by Claim~\ref{3cutsdts}. This leads to a valid orientation of $G$, a contradiction. Thus we may assume that $d\in G'-A$. If all such cuts separate $s$ from $q$, then $G'$ has a valid orientation by Claim~\ref{3cutsdts}. This leads to a valid orientation of $G$, a contradiction. Thus we may assume that $s\in G'-A$. In $G$, $\delta_G(G'-A)$ is not a $4$-edge-cut, else it separates $d$ and $s$ from $t$, a contradiction. Hence it is a $5$-edge-cut. Figure \ref{dtsseven} shows this graph.  \\

\begin{figure}
\begin{center}
\begin{tikzpicture}
\draw[dashed] (1.5,0) ellipse (0.8cm and 1.6cm) node {$A$};
\draw[dashed] (-1.5,0) ellipse (0.8cm and 1.6cm);
\draw (-1,0) -- (1,0);
\draw (-1,-0.5) -- (1,-0.5);
\draw (-1,-1) -- (1,-1);
\filldraw (-1.2,1.4) circle (2pt) node[above] {$q$};
\filldraw (1.2,1.4) circle (2pt) node[above] {$u$};
\filldraw (1.8,0.7) circle (2pt) node[above] {$v$};
\filldraw (1.5,1.05) circle (2pt) node[above] {$t$};
\filldraw (1.15,0.75) circle (2pt) node[below] {$w$};
\draw (1.2,1.4) -- (-1.2,1.4);
\draw (1.2,1.4) -- (-1.1,1.2);
\draw (1.2,1.4) -- (1.8,0.7);
\draw (1.5,1.05) -- (1.15,0.75);
\draw (1.2,1.4) -- (1.15,0.75);
\draw (1.8,0.7) -- (1.15,0.75);
\draw (1.8,0.7) -- (1.98,0.49);
\draw (1.8,0.7) -- (1.8,0.45);
\draw (1.15,0.75) -- (0.9,0.75);
\draw (1.15,0.75) -- (0.92,0.65);
\filldraw (-2.2,0.5) circle (2pt) node[left] {$s$};
\filldraw (-2.2,-0.5) circle (2pt) node[left] {$d$};
\draw[red] (1,0.9) to[out=-180,in=90] (0,-1.5);
\draw[red] (1,0.9) to[out=0,in=170] (2.5,0.8);
\end{tikzpicture}
\caption{\ref{boundary}: Reduce the $6$-edge-cut $\delta(A-\{u,t\})$.}
\label{dtsseven}
\end{center}
\end{figure}
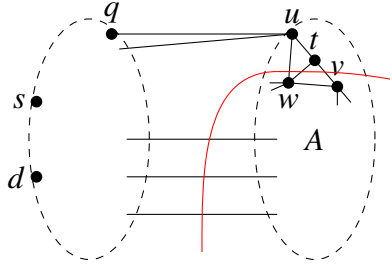

In $G$, contract $(A-\{r\})\cup\{v,w\}$ to a vertex, calling the resulting graph~$\bar{G}$. Then $\bar{G}$ is a DTS graph and has a valid orientation by the minimality of~$G$. Transfer this orientation to $G$, contract $(G'-A)\cup\{u,t\}$ to a vertex, and delete $tv$ and $tw$, calling the resulting graph~$\bar{G}'$. Then $\bar{G}'$ has a directed degree $4$ vertex $d'$ and one vertex of degree $3$~($v$). If $\bar{G}'$ contains a $2$-robust at most $2$-edge-cut, then $G$ contains a $2$-robust at most $4$-edge-cut that does not separate $d$ from $s$ and $t$, a contradiction. If $\bar{G}'$ contains a $2$-robust $3$-edge-cut, it necessarily separates $d'$ from $v$ (else $G$ contains an internal $2$-robust at most $4$-edge-cut), so $\bar{G}'$ has a valid orientation by Claim~\ref{3cutsdts}. Thus $\bar{G'}$ is a DTS graph and has a valid orientation by the minimality of~$G$. This leads to a valid orientation of $G$, a contradiction. Hence $G'$ has no $2$-robust at most $3$-edge-cuts. \\

Therefore $G'$ is a DTS graph and has a valid orientation by the minimality of~$G$. This leads to a valid orientation of $G$, a contradiction. The result follows.
\end{proof}

This concludes the proof of the properties of DTS graphs. Without loss of generality, the boundary of $F_G$ consists of the vertices $u,t,v=x,s,y,d$ in order. This graph is shown in Figure \ref{dtseight}. Let $A=\{u,t,v,s,y,d,w,z\}$. Then $\delta(A)$ is an internal $7$-edge-cut. If $G-A$ contains only one vertex, then the graph contains unoriented parallel edges, a contradiction. Hence $\delta(A)$ is $2$-robust. Contract $G-A$ to a vertex, calling the resulting graph~$G'$. Then $G'$ is a DTS graph and has a valid orientation by the minimality of~$G$. Transfer this orientation to~$G$. Contract $A$ to a single vertex $d'$ and delete the two edges incident with $w$, calling the resulting graph~$G''$. Then $d'$ is a directed vertex of degree $5$, and $G''$ contains no degree $3$ vertices, since in $G$ the neighbours of $w$ are distinct internal vertices of degree at least~$5$. If $G''$ contains a $2$-robust at most $3$-edge-cut, then $G$ contains a $2$-robust internal at most $5$-edge-cut, a contradiction. Hence $G''$ is a DTS graph and has a valid orientation by the minimality of~$G$. This leads to a valid orientation of $G$, a contradiction. Therefore no minimum counterexample exists, and Theorem \ref{maindts} follows. \end{proof}

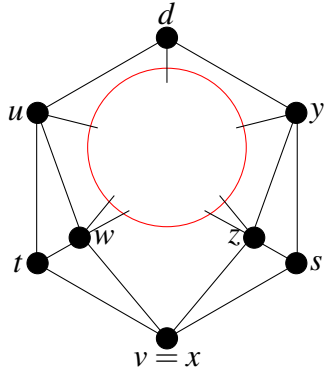
\begin{figure}
\begin{center}
\begin{tikzpicture}[scale = 2]
\begin{scope}[rotate=30]
\pgfkeys{/pgf/regular polygon sides=6,/pgf/minimum size=2cm}
\pgfnode{regular polygon}{center}{}{}{}
\pgfusepath{draw}
\end{scope}
\filldraw (0,-1) circle (2pt) node[below=1pt] {$v=x$};
\filldraw (0,1) circle (2pt) node[above=1pt] {$d$};
\filldraw (0.86,0.5) circle (2pt) node[right=1pt] {$y$};
\filldraw (0.86,-0.5) circle (2pt) node[right=1pt] {$s$};
\filldraw (-0.86,0.5) circle (2pt) node[left=1pt] {$u$};
\filldraw (-0.86,-0.5) circle (2pt) node[left=1pt] {$t$};

\filldraw (-0.58,-0.33) circle (2pt) node[right=1pt] {$w$};
\draw (-0.56,-0.33) -- (-0.86,-0.5);
\draw (-0.56,-0.33) -- (-0.86,0.5);
\draw (-0.56,-0.33) -- (0,-1);

\filldraw (0.58,-0.33) circle (2pt) node[left=1pt] {$z$};
\draw (0.56,-0.33) -- (0.86,-0.5);
\draw (0.56,-0.33) -- (0.86,0.5);
\draw (0.56,-0.33) -- (0,-1);

\draw[red] (0,0.27) circle (15pt);
\draw (0,1) -- (0,0.7);
\draw (-0.86,0.5) -- (-0.46,0.4);
\draw (0.86,0.5) -- (0.46,0.4);
\draw (-0.58,-0.33) -- (-0.35,-0.05);
\draw (-0.58,-0.33) -- (-0.25,-0.15);
\draw (0.58,-0.33) -- (0.35,-0.05);
\draw (0.58,-0.33) -- (0.25,-0.15);
\end{tikzpicture}
\caption{The boundary consists of only the vertices $u,t,v=x,s,y,d$.}
\label{dtseight}
\end{center}
\end{figure}

We conclude this section with some simple consequences of Theorem \ref{maindts}. In proving the Strong $3$-Flow Conjecture for projective planar graphs, our reductions will result in DTS graphs, and also graphs with three vertices of degree $3$ and no directed vertex. Theorem \ref{maindts} implies that such graphs also have a valid orientation.

\begin{definition}
An \emph{RST graph} is a graph $G$ embedded in the plane, together with a valid $\mathbb{Z}_3$-prescription function $p: V(G)\rightarrow\{-1,0,1\}$, such that:
\begin{enumerate}
\item $G$ is $3$-edge-connected,
\item $G$ has a specified face $F_G$, and at most three specified vertices $r$, $s$, and~$t$,
\item if $r$, $s$, and $t$ exist, then they have degree $3$ and are in the boundary of~$F_G$,
\item $G$ has at most three $3$-edge-cuts, which can only be $\delta(\{r\})$, $\delta(\{s\})$, and $\delta(\{t\})$, and
\item every vertex not in the boundary of $F_G$ has $5$ edge-disjoint paths to the boundary of~$F_G$.
\end{enumerate}
A \emph{3RST graph} is a graph $G$ with the above definition, where (4) is replaced by 
\begin{enumerate}
\item[4'.] all vertices aside from $r$, $s$, and $t$ have degree at least $4$, and if $r$, $s$, and $t$ exist, then every $3$-edge-cut in $G$ separates one of $r$, $s$, and $t$ from the other two.
\end{enumerate}
\end{definition}

The following three results are immediate consequences.

\begin{corollary}
Every RST graph has a valid orientation.
\label{mainrst}
\end{corollary}

\begin{lemma}
 \label{3dts}All 3DTS graphs have a valid orientation.
\end{lemma}

\begin{corollary}
 \label{3rst}All 3RST graphs have a valid orientation.
\end{corollary}

\section{Two Faces}

\label{twoface}

In this section we extend Theorem \ref{mainplanar} to allow low degree vertices on the boundaries of two specified faces of a graph. 

\begin{definition}
An \emph{FT graph} is a graph $G$ embedded in the plane, together with a valid prescription function $p: V(G)\rightarrow\{-1,0,1\}$, such that:
\begin{enumerate}
\item $G$ is $3$-edge-connected,
\item $G$ has two specified faces $F_G$ and $F_G^*$, and at most one specified vertex $d$ or~$t$,
\item there is at least one vertex in common between $F_G$ and~$F_G^*$,
\item if $d$ exists, then it has degree $3$, $4$, or $5$, is oriented, and is in the boundary of both $F_G$ and~$F_G^*$,
\item if $t$ exists, then it has degree $3$ and is in the boundary of at least one of $F_G$ and~$F_G^*$,
\item $G$ has at most one $3$-edge-cut, which can only be $\delta(\{d\})$ or $\delta(\{t\})$, and
\item every vertex not in the boundary of $F_G$ or $F_G^*$ has $5$ edge-disjoint paths to the union of the boundaries of $F_G$ and~$F_G^*$.
\end{enumerate}
We define all $3$-edge-connected graphs on at most two vertices to be FT graphs, regardless of vertex degrees. 
\end{definition}

We note that a DTS graph where at most one of $d$, $t$, and $s$ exists is an FT graph. 

\begin{theorem}
Every FT graph has a valid orientation. 
\label{mainft}
\end{theorem}
\begin{proof}

Let $G$ be a minimal counterexample with respect to the number of edges, followed by the number of unoriented edges. If $|E(G)|=0$, then $G$ consists of only an isolated vertex, and thus has a trivial valid orientation. If $|E(G)|-deg(d)=0$ then $G$ has an existing valid orientation. Thus we may assume $G$ has at least one unoriented edge. \\

We will establish the following series of properties of $G$. 
\begin{description}
\item[FT1:] The graph $G$ does not contain a loop, unoriented parallel edges, or a cut vertex.
\end{description}
We define a \emph{Type 1 cut} to be an edge-cut $\delta(A)$ that does not intersect the boundary of $F_G$ or~$F_G^*$. Since $F_G$ and $F_G^*$ have a common vertex, it follows that they are either both contained in $A$ or both contained in~$G-A$. Hence this is an internal cut. We define a \emph{Type 2 cut} to be an edge-cut $\delta(A)$ that intersects the boundary of exactly one of $F_G$ and~$F_G^*$. Finally, we define a \emph{Type 3 cut} to be an edge-cut $\delta(A)$ that intersects the boundary of both $F_G$ and~$F_G^*$.
\begin{description}
\item[FT2:] The graph $G$ does not contain
\begin{enumerate}
\item[a)] a $2$-robust at most $5$-edge-cut of Type 1 or 3,
\item[b)] a $2$-robust $4$-edge-cut,
\item[c)] a $2$-robust $5$-edge-cut of Type 2 where $t$ is on the side containing the boundary of both $F_G$ and $F_G^*$, or
\item[d)] a $2$-robust $6$-edge-cut of Type 1.
\end{enumerate}
\item[FT3:] The graph $G$ does not have a chord of the cycle bounding $F_G$, that is incident with a vertex of degree $3$ or $4$. 
\item[FT4:] The vertex $t$ exists. 
\item[FT5:] There is no edge in common between $F_G$ and $F_G^*$. 
\end{description}

Verifying these properties forms the bulk of the proof of Theorem \ref{mainft}. We complete the proof by considering a vertex in common between $F_G$ and $F_G^*$ of least degree.  \\

Property \ref{basicft} is straightforward, and the proof is omitted. It can be read in \citep{thesis}. 

\begin{Prop}{FT1}
The graph $G$ does not contain a loop, unoriented parallel edges, or a cut vertex. 
\label{basicft}
\end{Prop}

We consider edge-cuts in~$G$.

\begin{Prop}{FT2}
The graph $G$ does not contain
\begin{enumerate}
\item[a)] a $2$-robust at most $5$-edge-cut of Type 1 or 3, 
\item[b)] a $2$-robust $6$-edge-cut of Type 1,
\item[c)] a $2$-robust $4$-edge-cut, or
\item[d)] a $2$-robust $5$-edge-cut of Type 2 where $t$ is on the side containing the boundary of both $F_G$ and $F_G^*$.
\end{enumerate}
\label{cutsft}
\end{Prop}
\begin{proof}\mbox{}
\begin{enumerate}
\item[a)] Suppose that $G$ does contain a $2$-robust at most $5$-edge-cut $\delta_G(A)$ of Type 1 or~3. Assume that $d,t\not\in G-A$, and, if the cut is of Type 1, the boundaries of $F_G$ and $F_G^*$ are in~$A$.  Let $G'$ be the graph obtained from $G$ by contracting $G-A$ to a single vertex. The resulting vertex $v$ has degree $4$ or~$5$. If $\delta_G(A)$ is of Type 1, then $v$ has degree~$5$. If $v$ has degree $4$, then the cut is of Type 3, and $v$ is on the boundary of both specified faces in~$G'$. If $G'$ contains a $2$-robust cut $\delta_{G'}(B)$ of size at most $3$, then such a cut also exists in $G$, a contradiction. Hence $G'$ is an FT graph and has a valid orientation by the minimality of~$G$. Transfer this orientation to~$G$. \\

Let $G''$ be the graph obtained from $G$ by contracting $A$ to a single vertex~$v$. This vertex has degree $4$ or $5$ and is oriented. If $\delta_G(A)$ is of Type 3, then $v$ is on the boundary of both specified faces in~$G''$. If $\delta_G(A)$ is of Type 1, then we can choose both specified faces to be incident with~$v$. If $G''$ has a $2$-robust cut $\delta_{G''}(B)$ of size at most $3$, then such a cut also exists in $G$, a contradiction unless it is one of the specified vertices. Thus $G''$ is an FT graph and has a valid orientation by the minimality of~$G$. Transfer this orientation to $G$ to obtain a valid orientation of $G$, a contradiction.  

\item[b)] This case works in the same way as~a). In $G''$, there is a degree $6$ oriented vertex, and no degree $3$ or $4$ vertices. Hence we may delete one boundary edge incident with $v$ to obtain a graph $G'''$ with a degree $5$ oriented vertex and no degree $3$ vertex. If $G'''$ contains a $2$-robust cut of size at most $3$, then $G$ has a corresponding cut of size $4$, contradicting~a). Thus $G'''$ is an FT graph and has a valid orientation by the minimality of~$G$. This leads to a valid orientation of $G$, a contradiction. Hence no such cut exists.

\item[c)] Suppose that $G$ does contain a $2$-robust $4$-edge-cut $\delta_G(A)$. By a) $\delta_G(A)$ is of Type~2. Let $A$ be the side containing (part of) the boundaries of both $F_G$ and~$F_G^*$. Then $d\in A$ if it exists. Let $G'$ be the graph obtained from $G$ by contracting $G-A$ to a single vertex. The resulting vertex $v$ has degree $4$ and is on the boundary of a specified face. If $G'$ contains a $2$-robust cut $\delta_{G'}(B)$ of size at most $3$, then such a cut also exists in $G$, a contradiction unless it is one of the specified vertices. Hence $G'$ is an FT graph and has a valid orientation by the minimality of~$G$. Transfer this orientation to~$G$. \\

Let $G''$ be the graph obtained from $G$ by contracting $A$ to a single vertex~$v$. This vertex has degree $4$ and is oriented. There is only one specified face, which may contain~$t$. If $G''$ has a $2$-robust cut $\delta_{G''}(B)$ of size at most $3$, then such a cut also exists in $G$, a contradiction unless it is one of the specified vertices. Thus $G''$ is a DTS graph and has a valid orientation by Theorem~\ref{maindts}. Transfer this orientation to $G$ to obtain a valid orientation of $G$, a contradiction.  

\item[d)] This case works in the same way as~c). In $G''$ there is a degree $5$ oriented vertex and no degree $3$ vertex. Thus $G''$ is a DTS graph and has a valid orientation by Theorem~\ref{maindts}. Transfer this orientation to $G$ to obtain a valid orientation of $G$, a contradiction. \qedhere
\end{enumerate}
\end{proof}\medskip

We now consider the local properties of the graphs at vertices of low degree.

\begin{Prop}{FT3}
The graph $G$ does not have a chord of the cycle bounding $F_G$ that is incident with a vertex $u$ of degree $3$ or~$4$. 
\label{chordss}
\end{Prop}
\begin{proof}
Suppose that such a chord $uv$ exists, where $deg_G(u)\in\{3,4\}$. Let $H$ and $K$ be subgraphs of $G$ such that $H\cap K=\{\{u,v\},\{uv\}\}$, $H\cup K=G$, and $F_G^*$ is in~$H$. Note that this implies $d\in H$ if it exists.  \\

Suppose that $\delta(H)$ is not $2$-robust. Then $K$ contains $d$, else $G$ has unoriented parallel edges, and thus a valid orientation by \ref{basicft}. By definition, either $u$ or $v$ is~$d$. Suppose that $v=d$. Since $G$ has no unoriented degree $3$ vertex, $|\delta(H-\{u\})|\leq 3$, a contradiction. Now consider the case where $u=d$. Since $G$ has no unoriented parallel edges, $\delta(H)$ is a $3$-edge-cut and $G$ contains a degree $3$ vertex, a contradiction. Hence we may assume that $\delta(H)$ is $2$-robust.\\

Suppose that $\delta(K)$ is not $2$-robust. Then $|V(H)|=3$. If $u$ or $v$ is $d$, the above argument applies. Thus we may assume that $d$ is in $V(H)-\{u,v\}$. If there are parallel edges with endpoints $d$ and $u$, then $\delta(\{d,u\})$ is an at most $5$-edge-cut. Orient $u$ and contract the parallel edges between $d$ and $u$, calling the resulting graph~$G'$. Note that the vertex of contraction has the same degree as~$d$. Thus it is clear that $G'$ is an FT graph, and thus has a valid orientation by the minimality of~$G$. This leads to a valid orientation of $G$, a contradiction. Suppose that there are not parallel edges with endpoints $d$ and~$u$. Then there are parallel edges with endpoints $d$ and~$v$. Since $|\delta(\{d,v\})|\geq 5$, $deg_K(v)\geq 4$. Orient~$u$. Then $K$ is an FT graph, and has a valid orientation by the minimality of~$G$. This leads to a valid orientation of $G$, a contradiction. We may now assume that $\delta(K)$ is $2$-robust. \\

By definition, $\delta(H)$ and $\delta(K)$ have size at least~$4$. Hence $deg_G(v)\geq 6$, and so $v$ is not $d$ or~$t$. In addition, $v$ has degree at least $3$ in both $H$ and~$K$. \\

Suppose that $u\neq d$. Then in $H$, contract~$uv$. The graph $H/uv$ is an FT graph, and so by the minimality of $G$, $H/uv$ has a valid orientation. Transfer this orientation to $G$, and orient~$u$. In $K$, if $u$ has degree $2$, add an edge $e$ directed from $u$ to~$v$. Then $u$ is a directed vertex of degree~$3$. Since $F_G^*$ is in $H$, we can choose the second specified face to be incident with~$u$. Hence $K+e$ is an FT graph. By the minimality of $G$, $K'$ has a valid orientation. This leads to a valid orientation of $G$, a contradiction. \\

Thus $u=d$. Then in both $H$ and $K$ add a directed edge from $u$ to~$v$. It is clear that $H+e$ and $K+e$ are FT graphs, so by the minimality of $G$, they have valid orientations. Transfer the orientations of $H+e$ and $K+e$ to $G$ to obtain a valid orientation of $G$, a contradiction. Thus no such chord exists. 
\end{proof}

Similarly, no such chord of $F_G^*$ exists.

\begin{claim} 
The graph $G$ contains $d$ or $t$. 
\label{dort}
\end{claim}
\begin{proof}
Suppose for a contradiction that neither $d$ nor $t$ exists in~$G$. Let $v$ be a vertex in both $F_G$ and~$F_G^*$. If $deg(v)\leq 5$, orient $v$, calling the resulting graph~$G'$. Then it is clear that $G'$ is an FT graph, and has a valid orientation by the minimality of~$G$. This is a valid orientation of $G$, a contradiction. \\

We may assume $deg(v)\geq 6$. Orient and delete a boundary edge incident with $v$, calling the resulting graph~$G'$. Then $G'$ has at most one vertex of degree~$3$. If $G'$ contains a $2$-robust at most $3$-edge-cut, then $G$ contains a corresponding $2$-robust at most $4$-edge-cut, a contradiction. Hence $G'$ is an FT graph and has a valid orientation by the minimality of~$G$. This leads to a valid orientation of $G$, a contradiction.  \end{proof}

\begin{Prop}{FT4}
The vertex $t$ exists.
\end{Prop}
\begin{proof}
Suppose that $t$ does not exist. Then by Claim \ref{dort}, $d$ exists. By definition, $d$ is on the boundaries of both $F_G$ and~$F_G^*$. Suppose that $d$ has degree~$3$. Let $G'$ be the graph obtained from $G$ by deleting~$d$. Then $G'$ has at most three vertices of degree $3$ and no oriented vertex. If $G'$ contains a $2$-robust at most $3$-edge-cut, then $G$ contains a corresponding $2$-robust at most $4$-edge-cut, a contradiction. Hence $G'$ is an RST graph, and has a valid orientation by Corollary~\ref{mainrst}. This leads to a valid orientation of $G$, a contradiction. \\

Therefore, $d$ has degree $4$ or~$5$. Suppose that an edge $e$ incident with $d$ is in the boundary of $F_G$ and~$F_G^*$. Let $G'$ be the graph obtained from $G$ by deleting~$e$. Then $G'$ has at most one vertex of degree $3$ and an oriented vertex of degree $3$ or~$4$. If $G'$ contains a $2$-robust at most $3$-edge-cut, then $G$ contains a corresponding $2$-robust at most $4$-edge-cut, a contradiction. Hence $G'$ is a DTS graph, and has a valid orientation by Theorem~\ref{maindts}. This leads to a valid orientation of $G$, a contradiction. \\

We now assume that $F_G$ and $F_G^*$ do not have an edge in common incident with~$d$. Let $e_1$, $e_2$, $e_3$, $e_4$, and (possibly) $e_5$ be the edges incident with $d$ in cyclic order, where $e_1$ is on the boundary of $F_G$ and $e_2$ is on the boundary of~$F_G^*$. Let $G'$ be the graph obtained from $G$ by deleting $e_1$ and~$e_2$. Then $G'$ contains at most two vertices of degree $3$ and an oriented vertex of degree $2$ or~$3$. If $G'$ contains a $2$-robust at most $3$-edge-cut, then $G$ contains a corresponding $2$-robust either at most $4$-edge-cut, or $5$-edge-cut of Type 3, a contradiction. Let $G''$ be the graph obtained from $G'$ by adding a directed edge from $d$ to the other endpoint of $e_3$ if $deg_{G'}(d)=2$. Then $G'$ is a DTS graph, and has a valid orientation by Theorem~\ref{maindts}. This leads to a valid orientation of $G$, a contradiction. 
\end{proof}

\begin{claim}
The vertex $t$ is not in the boundary of both $F_G$ and $F_G^*$.
\label{notft}
\end{claim}
\begin{proof}
Suppose that $t$ is in the boundary of both $F_G$ and~$F_G^*$. Let $G'$ be the graph obtained from $G$ by orienting the edges incident with $t$ to satisfy $p(t)$. Then $G'$ is an FT graph and has a valid orientation by the minimality of $G$. This is a valid orientation of $G$, a contradiction. 
\end{proof}

\begin{Prop}{FT5}
There is no edge in common between $F_G$ and $F_G^*$. 
\label{nocommon}
\end{Prop}
\begin{proof}
Suppose for a contradiction that $G$ has an edge $e$ in the boundary of $F_G$ and~$F_G^*$. Then $e$ is not incident with $t$ by Claim~\ref{notft}. Let $G'$ be the graph obtained from $G$ by deleting~$e$. Then $G'$ has at most three vertices of degree $3$ and no oriented vertex. If $G'$ contains a $2$-robust at most $3$-edge-cut, then $G$ contains a corresponding $2$-robust at most $4$-edge-cut, a contradiction. Hence $G'$ is an RST graph, and has a valid orientation by Corollary~\ref{mainrst}. This leads to a valid orientation of $G$, a contradiction. 
\end{proof}

By definition, there exists a vertex in the boundaries of both $F_G$ and~$F_G^*$. Among all such vertices, let $v$ have the least degree, say~$k$. Let $e_1$, $e_2$, ..., $e_k$ be the edges incident with $v$ in cyclic order, where $e_1$ and $e_k$ are on the boundary of~$F_G$. Let $i$ be such that $e_i$, $e_{i+1}$ are on the boundary of~$F_G^*$.  Note that $i\neq 1$ and $i+1\neq k$ by \ref{nocommon}.\\

\begin{claim}
We have $k\geq 5$.
\end{claim}
\begin{proof}
The alternative is that $k=4$. Since $G$ does not contain unoriented parallel edges, at most one of the edges $e_1$, $e_2$, $e_3$, $e_4$ is incident with~$t$. Hence without loss of generality, we may assume that $e_1$ and $e_2$ are not incident with~$t$. Let $G'$ be the graph obtained from $G$ by lifting $e_3$ and $e_4$, and orienting and deleting $e_1$ and~$e_2$. Then $G'$ contains at most three vertices of degree~$3$. If $G'$ contains a $2$-robust at most $3$-edge-cut, then $G$ contains a corresponding $2$-robust at most $4$-edge-cut, or a $2$-robust at most $5$-edge-cut of Type 3, a contradiction. Then $G'$ is an RST graph, and has a valid orientation by Corollary~\ref{mainrst}. This leads to a valid orientation of $G$, a contradiction. 
\end{proof}

\begin{claim}
We have $i\neq 2$.
\end{claim}
\begin{proof}
Suppose that $i=2$. Let $G'$ be the graph obtained from $G$ by lifting $e_1$ and~$e_2$.  Then $G'$ contains at most two vertices of degree $3$: $v$ and~$t$. If $G'$ contains a $2$-robust at most $3$-edge-cut, then $G$ contains a corresponding $2$-robust at most $5$-edge-cut of Type 3, a contradiction. Then $G'$ is an RST graph, and has a valid orientation by Corollary~\ref{mainrst}. This leads to a valid orientation of $G$, a contradiction. Hence $i>2$. 
\end{proof}

Similarly we may assume that $i+1<k-1$. It follows that $k\geq 6$. \\

Let $G'$ be the graph obtained from $G$ by lifting $e_1$ and~$e_2$. Then $deg_{G'}(v)\geq 4$, so $G'$ has at most one degree $3$ vertex:~$t$. If $G'$ has a $2$-robust at most $2$-edge-cut, then $G$ has a corresponding $2$-robust at most $4$-edge-cut, a contradiction. If $G'$ is an FT graph and has a valid orientation by the minimality of $G$, this leads to a valid orientation of $G$, a contradiction. Hence we may assume that $G'$ is not an FT graph. Then $G'$ contains a $2$-robust $3$-edge-cut $\delta_{G'}(A)$. Now $\delta_G(A)$ is a $2$-robust $5$-edge-cut using $e_1$ and~$e_2$. By \ref{cutsft}, $\delta_G(A)$ is of Type 2 and separates $t$ from~$v$. \\

Let $G''$ be the graph obtained from $G$ by lifting $e_i$ and~$e_{i-1}$. Similarly, $G''$ is an FT graph and has a valid orientation by the minimality of $G$, unless $G''$ has a $2$-robust $5$-edge-cut $\delta_{G''}(B)$ of Type 2 using $e_i$ and $e_{i-1}$ that separates $t$ from~$v$.\\

But $\delta_G(A)$ intersects only the boundary of $F_G$, so $t$ is on the boundary of $F_G$ and not~$F_G^*$. Similarly $\delta_G(B)$ intersects only the boundary of $F_G^*$, so $t$ is on the boundary of $F_G^*$ and not $F_G$, a contradiction. \\

We conclude that no such counterexample exists. Therefore all FT graphs have a valid orientation. \end{proof} 

\section{Discussion}

\label{projdis}

In this section we relate the theorems of this paper to the $3$-Flow Conjecture and the Strong $3$-Flow Conjecture, and consider possible extensions of these results.\\

Theorem \ref{mainplanar} allows a directed vertex of degree $5$, or a directed vertex of degree $4$ and a degree $3$ vertex, but not both a directed vertex of degree $5$ and a degree $3$ vertex. The reason for this is that any graph with a degree $5$ directed vertex $d$ adjacent via parallel edges to a vertex $t$ of degree $3$ does not have a valid orientation for some prescription functions. For example, if the edges incident with $d$ and $t$ are directed into $t$, and $p(t)=-1$, then $G$ has no valid orientation. Such a graph can be seen in Figure~\ref{d5}. \\

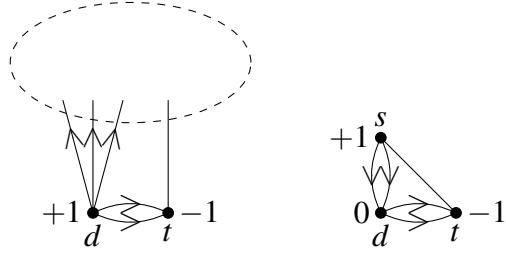
\begin{figure}
\begin{center}
\subfloat{
\begin{tikzpicture}
\draw[dashed] (0,1.5) ellipse (1.6cm and 0.8cm);
\filldraw (-0.5,-0.5) circle (2pt) node[below] {$d$};
\filldraw (0.5,-0.5) circle (2pt) node[below] {$t$};
\filldraw (-0.5,-0.5) circle (0pt) node[left] {$+1$};
\filldraw (0.5,-0.5) circle (0pt) node[right] {$-1$};
\draw (0.5,-0.5) -- (0.5,1);
\draw (-0.5,-0.5) -- (-0.5,1);
\draw (-0.5,-0.5) -- (-0.9,1);
\draw (-0.5,-0.5) -- (-0.1,1);
\draw (-0.5,-0.5) to[out=25,in=155] (0.5,-0.5);
\draw (-0.5,-0.5) to[out=-25,in=-155] (0.5,-0.5);
\filldraw (0,-0.38) circle (0pt) node {$>$};
\filldraw (0,-0.62) circle (0pt) node {$>$};
\filldraw (-0.5,0.5) circle (0pt) node[rotate=90] {$>$};
\filldraw (-0.23,0.5) circle (0pt) node[rotate=77] {$>$};
\filldraw (-0.77,0.5) circle (0pt) node[rotate=103] {$>$};
\end{tikzpicture}
}
\subfloat{
\begin{tikzpicture}
\filldraw (-0.5,-0.5) circle (2pt) node[below] {$d$};
\filldraw (0.5,-0.5) circle (2pt) node[below] {$t$};
\filldraw (-0.5,-0.5) circle (0pt) node[left] {$0$};
\filldraw (0.5,-0.5) circle (0pt) node[right] {$-1$};
\filldraw (-0.5,0.5) circle (2pt) node[above] {$s$};
\filldraw (-0.5,0.5) circle (0pt) node[left] {$+1$};
\draw (0.5,-0.5) -- (-0.5,0.5);
\draw (-0.5,-0.5) to[out=25,in=155] (0.5,-0.5);
\draw (-0.5,-0.5) to[out=-25,in=-155] (0.5,-0.5);
\draw (-0.5,-0.5) to[out=65,in=-65] (-0.5,0.5);
\draw (-0.5,-0.5) to[out=115,in=-115] (-0.5,0.5);
\filldraw (0,-0.38) circle (0pt) node {$>$};
\filldraw (0,-0.62) circle (0pt) node {$>$};
\filldraw (-0.62,0) circle (0pt) node[rotate=90] {$<$};
\filldraw (-0.38,0) circle (0pt) node[rotate=90] {$<$};
\draw (-2,0) circle (0pt);
\end{tikzpicture}
}
\caption{A graph with a directed vertex of degree $5$ ($4$), one (two) degree $3$ vertex (vertices), and no valid orientation.}
\label{d5}
\end{center}
\end{figure}

In Theorem \ref{maindts} we allow a directed vertex of degree $3$ with two other degree $3$ vertices. Again, a directed vertex of degree $4$ with two degree $3$ vertices is not possible. If $d$ is a degree $4$ directed vertex adjacent via parallel edges to a vertex $t$ of degree $3$, then there may not be an orientation of $t$ that extends the existing orientation of $d$ and meets~$p(t)$. Now $\delta(\{d,t\})$ is a $3$-edge-cut, but since the graph has a second degree $3$ vertex, such a $3$-edge-cut need not be $2$-robust. This graph can also be seen in Figure~\ref{d5}. \\

When increasing the number of unoriented degree $3$ vertices to three, we know of no graph or family of graphs that would rule out a directed vertex of degree~$3$. While the example in Figure \ref{d5} does not extend to this case, we note that there is an infinite family of graphs with three degree $3$ vertices and an oriented degree $4$ vertex that do not have a valid prescription. Let $G$ be a graph where the boundary of the outer face consists of a directed degree $4$ vertex $d$, and three degree $3$ vertices $r$, $s$, and $t$. Let $p(d)=p(t)=-1$, $p(r)=p(s)=0$, and assume that all edges incident with $d$ are directed out from~$d$. Let $A=G-\{d,r,s,t\}$. Then $\delta(A)$ is an internal $5$-edge-cut. We assume that $p(A)=-1$. Then $p(G)=0$, so $p$ is a valid prescription function. Since $rd$ is directed into $r$, all edges incident with $r$ must be directed into~$r$. Since $rs$ is directed out of $s$, all edges incident with $s$ must be directed out of~$s$. Then $st$ and $dt$ are directed into $t$. No direction of the remaining edge incident with $t$ meets $p(t)$, so $G$ does not have a valid orientation that meets~$p$. This family of graphs can be seen in Figure~\ref{33s}.\\

\begin{figure}
\begin{center}
\subfloat{
\begin{tikzpicture}[scale=2]
\draw[dashed] (0,0) circle (10pt) node {$-1$};
\filldraw (0,1) circle (2pt) node[above=1pt] {$d,-1$};
\filldraw (1,0) circle (2pt) node[right=1pt] {$r,0$};
\filldraw (0,-1) circle (2pt) node[below=1pt] {$s,0$};
\filldraw (-1,0) circle (2pt) node[left=1pt] {$t,-1$};

\draw (1,0) -- (0,1);
\draw (1,0) -- (0,-1);
\draw (-1,0) -- (0,1);
\draw (-1,0) -- (0,-1);
\draw (0,-1) -- (0,-0.3);
\draw (1,0) -- (0.3,0);
\draw (-1,0) -- (-0.3,0);
\draw (0,1) to[out=-65,in=65] (0.1,0.3);
\draw (0,1) to[out=-115,in=115] (-0.1,0.3);

\draw (0.5,0.5) circle (0pt) node[rotate=-45] {$>$};
\draw (-0.5,0.5) circle (0pt) node[rotate=-135] {$>$};
\draw (0.15,0.6) circle (0pt) node[rotate=-90] {$>$};
\draw (-0.15,0.6) circle (0pt) node[rotate=-90] {$>$};
\end{tikzpicture}
}
\subfloat{
\begin{tikzpicture}[scale = 2]
\draw[dashed] (0,0) circle (10pt) node {$0$};
\filldraw (0,1) circle (2pt) node[above=1pt] {$d,0$};
\filldraw (0.95,0.31) circle (2pt) node[right=1pt] {$r,0$};
\filldraw (-0.95,0.31) circle (2pt) node[left=1pt] {$u,0$};
\filldraw (0.59,-0.81) circle (2pt) node[right=1pt] {$s,0$};
\filldraw (-0.59,-0.81) circle (2pt) node[left=1pt] {$t,0$};

\draw (0,1) -- (0,0.3);
\draw (0.95,0.31) -- (0.3,0.1);
\draw (-0.95,0.31) -- (-0.3,0.1);
\draw (0.59,-0.81) -- (0.18,-0.24);
\draw (-0.59,-0.81) -- (-0.18,-0.24);

\draw (0,1) -- (0.95,0.31);
\draw (0.95,0.31) -- (0.59,-0.81);
\draw (0,1) -- (-0.95,0.31);
\draw (-0.95,0.31) -- (-0.59,-0.81);
\draw (0.59,-0.81) -- (-0.59,-0.81);

\draw (0,0.7) circle (0pt) node[rotate=90] {$<$};
\draw (0.48,0.65) circle (0pt) node[rotate=-36] {$>$};
\draw (-0.48,0.65) circle (0pt) node[rotate=-144] {$>$};

\draw (0,-1.3) circle (0pt);
\end{tikzpicture}
}
\caption{A graph with a directed vertex of degree $4$ ($3$), three (four) degree $3$ vertices, and no valid orientation.}
\label{33s}
\end{center}
\end{figure}
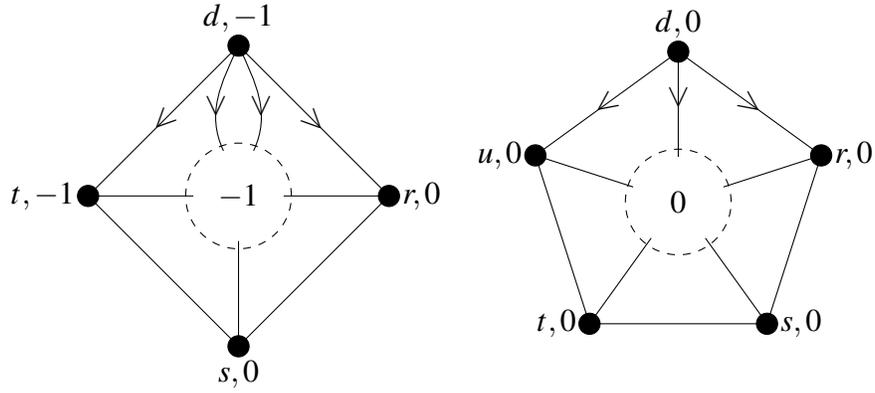

If we allow a directed vertex $d$ of degree $3$ and four vertices of degree $3$, then we obtain a similar family of graphs. Let the boundary of the outer face of $G$ consist of $d$ and the four vertices of degree $3$, producing an internal $5$-edge-cut. Assume that all vertices have prescription zero. Each of the five vertices on the boundary of the outer face has either all edges pointing into the vertex, or all edges pointing out. It is clear that with an odd length boundary, this is not possible. Hence $G$ does not have a 3-flow (as opposed to simply a modulo $3$ orientation meeting~$p$). This family of graphs can also be seen in Figure~\ref{33s}. \\

Similarly, we consider extending Corollary \ref{mainrst} to allow further degree $3$ vertices. It is clear that a family of graphs with five vertices of degree $3$ similar to that with four and a directed degree $3$ vertex can be constructed, that do not have nowhere zero $\mathbb{Z}_3$-flows. We conjecture that such graphs with four degree $3$ vertices or three degree $3$ vertices and a directed degree $3$ vertex have a valid orientation to meet a given valid prescription function. If true, then this would be a best possible such result. 

\begin{conjecture}
Let $G$ be a graph embedded in the plane, together with a valid prescription function $p: V(G)\rightarrow\{-1,0,1\}$ such that:
\begin{enumerate}
\item $G$ is $3$-edge-connected,
\item $G$ has a specified face $F_G$, and at most four specified vertices $d$, $r$, $s$, and $t$, 
\item if $d$ exists, then it has degree $3$, is in the boundary of $F_G$, and may be oriented,
\item if $r$, $s$, or $t$ exists, then it has degree $3$ and is in the boundary of $F_G$,
\item $G$ has at most four $3$-edge-cuts, which can only be $\delta(\{d\})$, $\delta(\{r\})$, $\delta(\{s\})$, and $\delta(\{t\})$, and
\item every vertex not in the boundary of $F_G$ has $5$ edge-disjoint paths to the boundary of $F_G$.
\end{enumerate}
Then $G$ has a valid orientation. 
\end{conjecture}

Theorem \ref{mainft} requires $d$, if it exists, to be in the boundaries of both $F_G$ and $F_G^*$. In the following conjecture we remove this hypothesis, and the requirement that the two specified faces have a vertex in common. 

\begin{conjecture}
Let $G$ be a graph embedded in the plane, together with a valid prescription function $p: V(G)\rightarrow\{-1,0,1\}$, such that:
\begin{enumerate}
\item $G$ is $3$-edge-connected,
\item $G$ has two specified faces $F_G$ and $F_G^*$, and at most one specified vertex $d$ or $t$,
\item if $d$ exists, then it has degree $3$, $4$, or $5$, is oriented, and is in the boundary of $F_G$ or $F_G^*$,
\item if $t$ exists, then it has degree $3$ and is in the boundary of $F_G$ or $F_G^*$,
\item $G$ has at most one $3$-edge-cut, which can only be $\delta(\{d\})$ or $\delta(\{t\})$, and
\item every vertex not in the boundary of $F_G$ or $F_G^*$ has $5$ edge-disjoint paths to the union of the boundaries of $F_G$ and $F_G^*$.
\end{enumerate}
Then $G$ has a valid orientation. 
\label{conj4}
\end{conjecture}

To adapt the arguments from [Insert Citation] to the torus would require a version of Conjecture \ref{conj4} in which both $d$ and $t$ exist. The following example shows that this version does not hold. Consider a graph $G$ with two specified faces $F_G$ and $F_G^*$ such that every vertex is on the boundary of $F_G$ or~$F_G^*$. Let $t$ be a degree $3$ vertex on the boundary of $F_G^*$, and $w$ a degree $5$ vertex adjacent to $t$ on the boundary of~$F_G$. Suppose that $G$ has a cycle $P=wv_0v_1...v_nw$ where $n=6k$ for some $k\in \mathbb{Z}^+$, $v_i$ has degree $4$ for all $1\leq i\leq n$, $v_1$, $v_2$, ..., $v_n$ alternate between the boundaries of $F_G$ and $F_G^*$, and $V(G)=\{t,w,v_1,v_2,...,v_n\}$. Assume that $d=v_{\frac{n}{2}}$. We set $w=v_{-1}=v_{n+1}$ and $t=v_{-2}=v_{n+2}$. See Figure~\ref{star}.\\

\begin{figure}[t]
\begin{center}
\begin{tikzpicture}[scale=1.4]
\draw (0,0) circle (50pt) node {$F_G$};
\draw (0,0) circle (100pt);
\draw (4,0) circle (0pt) node {$F_G^*$};

\filldraw (0,1.75) circle (2pt) node[below] {$t, 0$};
\filldraw (0,3.5) circle (2pt) node[above] {$w,0$};
\filldraw (1.52,0.88) circle (2pt) node[xshift=-0.7cm,yshift=-0.2cm] {$v_2,+1$};
\filldraw (0.88,1.52) circle (2pt) node[xshift=-0.3cm,yshift=-0.4cm] {$v_0,+1$};
\filldraw (0.88,-1.52) circle (2pt) node[xshift=-0.3cm,yshift=0.6cm] {$v_{3k-2},+1$};
\filldraw (0,-1.75) circle (2pt) node[above] {$d, -1$};
\filldraw (-1.52,0.88) circle (2pt) node[xshift=0.7cm,yshift=-0.3cm] {$v_{6k-2},-1$};
\filldraw (-0.88,1.52) circle (2pt) node[xshift=0.3cm,yshift=-0.4cm] {$v_{6k},-1$};
\filldraw (-0.88,-1.52) circle (2pt) node[xshift=0.3cm,yshift=0.6cm] {$v_{3k+2},+1$};
\filldraw (2.48,2.48) circle (2pt) node[right] {$v_1,+1$};
\filldraw (-2.48,2.48) circle (2pt) node[left] {$v_{6k-1},-1$};
\filldraw (0.91,-3.38) circle (2pt) node[below] {$v_{3k-1},+1$};
\filldraw (2.48,-2.48) circle (2pt) node[right] {$v_{3k-3},+1$};
\filldraw (-0.91,-3.38) circle (2pt) node[below] {$v_{3k+1},+1$};
\filldraw (-2.48,-2.48) circle (2pt) node[left] {$v_{3k+3},-1$};

\draw (0,3.5) -- (0,1.75);
\draw (0,3.5) -- (0.88,1.52);
\draw (0.88,1.52) -- (2.48,2.48);
\draw (2.48,2.48) -- (1.52,0.88);
\draw (1.52,0.88) -- (2.45,0.895);
\draw (0,3.5) -- (0,1.75);
\draw (0,3.5) -- (-0.88,1.52);
\draw (-0.88,1.52) -- (-2.48,2.48);
\draw (-2.48,2.48) -- (-1.52,0.88);
\draw (-1.52,0.88) -- (-2.45,0.895);
\draw (0,-1.75) -- (0.91,-3.38);
\draw (0.91,-3.38) -- (0.88,-1.52);
\draw (0.88,-1.52) -- (2.48,-2.48);
\draw (2.48,-2.48) -- (2,-1.68);
\draw (0,-1.75) -- (-0.91,-3.38);
\draw (-0.91,-3.38) -- (-0.88,-1.52);
\draw (-0.88,-1.52) -- (-2.48,-2.48);
\draw (-2.48,-2.48) -- (-2,-1.68);

\draw (2.25,0) circle (0pt) node {\vdots};
\draw (-2.25,0) circle (0pt) node {\vdots};

\draw (0.44,-1.7) circle (0pt) node[rotate=13] {$>$};
\draw (-0.44,-1.7) circle (0pt) node[rotate=167] {$>$};
\draw (0.46,-2.57) circle (0pt) node[rotate=-60] {$>$};
\draw (-0.46,-2.57) circle (0pt) node[rotate=-120] {$>$};
\end{tikzpicture}
\caption{A graph with two specified faces, a directed vertex of degree $4$, and no valid orientation.}
\label{star}
\end{center}
\end{figure}

Suppose that $p(d)=-1$ and all edges incident with $d$ are directed out from~$d$. Let $p(v_{\frac{n}{2}-1})=p(v_{\frac{n}{2}-2})=p(v_{\frac{n}{2}+1})=p(v_{\frac{n}{2}+2})=+1$. Let $p(v_j)=+1$ for $j<\frac{n}{2}-2$ and $p(v_j)=-1$ for $j>\frac{n}{2}+2$. Let $p(t)=p(w)=0$. Then it is clear that $p$ is a valid prescription function for~$G$. \\

\begin{lemma}
The graph $G$ does not have a valid orientation that meets $p$.
\label{star}
\end{lemma}
\begin{proof}
We first show that for all $0\leq j\leq k$, $v_{3(k-j)}v_{3(k-j)-1}$ and $v_{3(k-j)}v_{3(k-j)-2}$ point out of $v_{3(k-j)}$. We proceed by induction on~$j$. The base case is determined by the directions of the edges incident with~$d$. Suppose that for some $j$ where $0\leq j\leq k-1$, $v_{3(k-j)}v_{3(k-j)-1}$ and $v_{3(k-j)}v_{3(k-j)-2}$ point out of~$v_{3(k-j)}$. We consider $v_{3(k-j-1)}$. Since $v_{3(k-j)}v_{3(k-j)-1}$ and $v_{3(k-j)}v_{3(k-j)-2}$ point out from $v_{3(k-j)}$, the remaining edges at each of $v_{3(k-j)-1}$ and $v_{3(k-j)-2}$ point all in or all out. Since these vertices are adjacent, $v_{3(k-j)-3}v_{3(k-j)-1}$ and $v_{3(k-j)-3}v_{3(k-j)-2}$ point one in and one out of $v_{3(k-j)-3}$. Hence the remaining two edges at $v_{3(k-j)-3}$: $v_{3(k-j)-4}v_{3(k-j)-4}$ and $v_{3(k-j)-3}v_{3(k-j)-5}$, must satisfy the prescription of $v_{3(k-j)-3}$ and thus point out of $v_{3(k-j)-3}=v_{3(k-j-1)}$ as required. \\

Similarly for all $0\leq j\leq k$, $v_{3(k+j)}v_{3(k+j)+1}$ and $v_{3(k+j)}v_{3(k+j)+2}$ point into~$v_{3(k+j)}$. Then $v_{6k}t$ points out of $t$ and $v_{0}t$ points into~$t$. Hence no direction of $tw$ meets~$p(t)$. Thus $G$ has no valid orientation that meets~$p$.
\end{proof}

Such a result may be possible if $d$ is restricted to degree~$3$.

\bibliography{ResearchProposal}
\bibliographystyle{myapalike}

\end{document}